\documentclass[12pt]{amsart}
\usepackage{txfonts}
\usepackage{amscd,amssymb,mathabx,subfigure}
\usepackage[all]{xy}
\usepackage{graphicx,calrsfs}
\usepackage[colorlinks,plainpages,backref,urlcolor=blue,breaklinks]{hyperref}
\usepackage{tikz}
\usetikzlibrary{calc}
\usepackage{mdwlist}
\usepackage{array}
\usepackage{scalerel}

\newtheorem{theorem}{Theorem}[section]
\newtheorem{lemma}[theorem]{Lemma}
\newtheorem{corollary}[theorem]{Corollary}
\newtheorem{proposition}[theorem]{Proposition}

\theoremstyle{definition}

\newtheorem{example}[theorem]{Example}
\newtheorem{remark}[theorem]{Remark}

\newtheorem{problem}[theorem]{Problem}

\newcommand{\Z}{\mathbb{Z}}
\newcommand{\Q}{\mathbb{Q}}
\newcommand{\R}{\mathbb{R}}
\newcommand{\C}{\mathbb{C}}

\newcommand{\CP}{\mathbb{CP}}

\renewcommand{\k}{\Bbbk}
\newcommand{\one}{\mathbf{1}}
\newcommand{\zero}{\mathbf{0}}

\newcommand{\RR}{{\mathcal R}}
\newcommand{\VV}{\mathcal{V}}

\newcommand{\ZZ}{\mathcal{Z}}

\newcommand{\wV}{\mathcal{W}}
\newcommand{\PP}{\mathcal{P}}

\newcommand{\A}{{\mathcal{A}}}

\newcommand{\m}{{\mathfrak{m}}}

\DeclareMathOperator{\D}{d}
\DeclareMathOperator{\init}{in}
\DeclareMathOperator{\rank}{rank}

\DeclareMathOperator{\coker}{coker}

\DeclareMathOperator{\ab}{{ab}}
\DeclareMathOperator{\Sym}{Sym}

\DeclareMathOperator{\GL}{GL}

\DeclareMathOperator{\Hom}{{Hom}}

\DeclareMathOperator{\TC}{TC}

\DeclareMathOperator{\Char}{Char}

\DeclareMathOperator{\lk}{lk}

\DeclareMathOperator{\pf}{pf}
\DeclareMathOperator{\Pf}{Pf}
\DeclareMathOperator{\Det}{Det}
\DeclareMathOperator{\tor}{Tors}

\DeclareMathOperator{\Ann}{ann}
\DeclareMathOperator{\Spec}{Spec}

\DeclareMathOperator{\PD}{PD}

\newcommand{\myempty}{\text{\O}}
\newcommand{\dR}{\scriptscriptstyle{\rm dR}}
\newcommand{\PL}{\scriptscriptstyle{\rm PL}}
\DeclareMathOperator*{\connsum}{\scalerel*{\#}{\bigoplus}}

\newcommand{\cdga}{\textsc{cdga}}
\newcommand{\cga}{\textsc{cga}}
\newcommand{\surj}{\twoheadrightarrow}

\def\set#1{{\left\{#1\right\}}}

\def\dot{\mathchar"013A}  
\newcommand{\hdot}{{\raise1pt\hbox to0.35em{\Large $\dot$\!}}} 
\newcommand{\bwedge}{\mbox{\normalsize $\bigwedge$}}

\newcommand{\cz}{\check{{\mathcal Z}}}

\numberwithin{equation}{section}

\makeatletter
\def\@tocline#1#2#3#4#5#6#7{\relax
  \ifnum #1>\c@tocdepth 
  \else
    \par \addpenalty\@secpenalty\addvspace{#2}%
    \begingroup \hyphenpenalty\@M
    \@ifempty{#4}{%
      \@tempdima\csname r@tocindent\number#1\endcsname\relax
    }{%
      \@tempdima#4\relax
    }%
    \parindent\z@ \leftskip#3\relax \advance\leftskip\@tempdima\relax
    \rightskip\@pnumwidth plus4em \parfillskip-\@pnumwidth
    #5\leavevmode\hskip-\@tempdima
      \ifcase #1
       \or\or \hskip 1em \or \hskip 2em \else \hskip 3em \fi%
      #6\nobreak\relax
    \dotfill\hbox to\@pnumwidth{\@tocpagenum{#7}}\par
    \nobreak
    \endgroup
  \fi}
\makeatother

\title[Cohomology jump loci of $3$-manifolds]%
{Cohomology jump loci of $3$-manifolds}

\author[Alexander~I.~Suciu]{Alexander~I.~Suciu$^1$}
\address{Department of Mathematics,
Northeastern University,
Boston, MA 02115, USA}
\email{\href{mailto:a.suciu@northeastern.edu}{a.suciu@northeastern.edu}}
\urladdr{\href{http://web.northeastern.edu/suciu/}%
{web.northeastern.edu/suciu/}}
\thanks{$^1$This work was supported by Simons Foundation Collaboration 
Grants for Mathematicians \#354156 and \#693825}



\begin{document}

\begin{abstract}
The cohomology jump loci of a space $X$ are of two 
basic types: the characteristic varieties, 
defined in terms of homology with coefficients in rank 
one local systems, and the resonance varieties, 
constructed from information encoded in either the cohomology ring, 
or an algebraic model for $X$. We explore here the geometry of 
these varieties and the delicate interplay between them 
in the context of closed, orientable $3$-dimensional 
manifolds and link complements.  The classical multivariable 
Alexander polynomial plays an important role in this analysis. 
As an application, we derive some consequences regarding 
the formality and the existence of finite-dimensional models 
for such $3$-manifolds.
\end{abstract}

\maketitle
\tableofcontents

\section{Introduction}
\label{sect:intro}

\subsection{Cohomology jump loci}
\label{intro:cjl}

Let $X$ be a finite, connected CW-complex 
and let $\pi=\pi_1(X)$ be its fundamental group. 
The {\em characteristic varieties}\/ $\VV^i_k(X)$ are the 
Zariski closed subsets of the algebraic group $\Hom(\pi,\C^{*})$ 
consisting of those characters $\rho\colon \pi\to \C^{*}$  
for which the $i$-th homology group of $X$ 
with coefficients in the rank~$1$ local system defined by $\rho$ 
has dimension  at least $k$; in particular, the trivial character 
$\one$ belongs to $\VV^i_k(X)$ precisely when the $i$-th Betti 
number $b_i(X)$ is at least $k$.

Now let $H^{\hdot}=H^{\hdot}(X,\C)$ be the cohomology algebra of $X$. 
For each $a\in H^1$, we may form a cochain complex, $(H,a),$ with 
differentials $\delta_a\colon H^i\to H^{i+1}$ given by 
left-multiplication by $a$. 
The {\em resonance varieties}\/ $\RR^i_k(X)$, then, 
are the subvarieties of the affine space 
$H^1$ consisting of those classes $a$ for 
which the $i$-th cohomology of $(H,a)$ has dimension 
at least $k$.

Finally, suppose we are given an algebraic model for $X$, that is, a 
commutative differential graded algebra $(A,\D)$ connected by a 
zig-zag of quasi-isomorphisms to the Sullivan algebra of polynomial 
forms on $X$. Assuming $A$ is connected and of finite type,  
we may form a cochain complex $(A,\delta_a)$ as above, with 
differentials now given by $\delta_a(u)=au+\D{u}$, and we may 
define the resonance varieties $\RR^i_k(A)\subseteq H^1(A)$ 
analogously.   (When if $X$ is formal, that is, the cohomology 
algebra $H^{\hdot}(X,\C)$ with $\D=0$ is a model for $X$, we 
recover the previous definition of resonance.) 

All these notions admit `partial' versions: e.g., for a fixed $q\ge 1$, 
we may speak of a $q$-finite $q$-model $(A,\D)$ for $X$, 
in which case the sets $\RR^i_k(A)$ are Zariski closed 
for all $i\le q$. For more details on all this, we refer to 
\cite{DS18, DP-ccm, DPS-duke, MPPS, Su-tcone} 
and references therein.

For $q=1$, the aforementioned properties of the space $X$ can  
be interpreted purely in terms of the Malcev Lie algebra of it fundamental 
group, $\m(\pi)$.  For instance, as shown in \cite{PS-jlms}, $X$ admits a 
$1$-finite $1$-model if and only if $\m(\pi)$ is the lower central series 
completion of a finitely presented Lie algebra $L$.  More stringently, 
as shown in the foundational work of Quillen \cite{Qu} and 
Sullivan \cite{Su75}, $X$ is $1$-formal if and only if $L$ can 
be chosen to be a quadratic Lie algebra.  

\subsection{The Tangent Cone formula}
\label{intro:tc}

A crucial tool in both the theory and the applications of cohomology jump 
loci is a formula relating the behavior around the origin of the characteristic 
and resonance varieties of a space. 

Given a subvariety $W\subseteq (\C^*)^n$, we consider two types 
of approximations around the trivial character.  One is the usual 
tangent cone, $\TC_{\one}(W)\subseteq \C^n$, while the other is the 
exponential tangent cone, $\tau_{\one}(W)$, which consists of those 
$z\in \C^n$ for which $\exp(\lambda z)\in W$, for all $\lambda\in \C$. 
As shown by Dimca--Papadima--Suciu in \cite{DPS-duke}, $\tau_{\one}(W)$ 
is a finite union of rationally defined linear subspaces, all contained in 
$\TC_{\one}(W)$.  

Now let $X$ be a space as above. Combining the previous 
observation together with a result of Libgober \cite{Li02} 
yields a chain of inclusions, 
\begin{equation}
\label{eq:tc inc-intro}
\tau_{\one}(\VV^i_k(X))\subseteq  
\TC_{\one}(\VV^i_k(X))\subseteq \RR^i_k(X),
\end{equation}
in arbitrary degree $i$ and depth $k$.  As we shall see, each of these 
inclusions may be strict. Nevertheless, if $X$ admits a $q$-finite 
$q$-model $A$,  it follows from work of Dimca--Papadima \cite{DP-ccm} 
and Budur--Wang \cite{BW17} that the following ``Tangent Cone formula" holds:
\begin{equation}
\label{eq:tc-fm-intro}
\tau_{\one}(\VV^i_k(X))=\TC_{\one}(\VV^i_k(X))=\RR^i_k(A) 
\end{equation}
for all $i\le q$ and $k\ge 0$.
In particular, if $X$ is $q$-formal, then, in the same range,
\begin{equation}
\label{eq:tcone-intro}
\tau_{\one}(\VV^i_k(X))=\TC_{\one}(\VV^i_k(X))=\RR^i_k(X),
\end{equation}
a result originally proved in \cite{DPS-duke} for $i=1$.

\subsection{Cohomology jump loci of closed $3$-manifolds}
\label{intro:res3m}

Most of the applications of these results have centered on 
the case when $X$ admits a finite-dimensional model, 
which happens for instance if $X$ is a smooth, quasi-projective 
variety (in particular, the complement of a hyperplane arrangement),  
or a compact K\"ahler manifold, or a Sasakian manifold, 
or a nilmanifold, or a classifying space for a right-angled 
Artin group.  

We focus here instead on the cohomology jump loci of 
$3$-dimen\-sional manifolds, which in general fail to possess  
finite-dimensional models.  Let $M$ be a compact, connected 
$3$-manifold; we shall assume for simplicity that $M$ is 
orientable and $\partial M = \myempty$, although we shall 
also treat in \S\ref{sect:links} the case when $M$ is a link complement. 
Set $n=b_1(M)$.   Sending each element of $\pi=\pi_1(M)$ to its inverse 
induces an automorphism of the character group of $\pi$, 
which in turn restricts to isomorphisms  
$\VV^{i}_k(M)\cong \VV^{3-i}_k(M)$.  Thus, in order 
to compute the characteristic varieties of $M$, it is enough 
to determine the jump loci $\VV^1_k(M)$ for $1\le k\le n$.  

Work of McMullen \cite{McM} and Turaev \cite{Tu} implies that, 
at least away from the origin $\one$, 
the intersection of $\VV^1_1(M)$ with the identity component of the 
character group coincides with $V(\Delta_M)$, the hypersurface 
defined by the Alexander polynomial $\Delta_M\in \Z[t_1^{\pm 1},\dots , t_n^{\pm 1}]$. 
It follows that $\TC_{\one}(\VV^1_1(M))$ is either $\{\zero\}$, or $\C^n$, or 
the subvariety of $\C^n$ defined defined by the initial form of 
the polynomial $\left.\Delta_M\right|_{t_i-1=x_i} \in \Z[x_1,\dots, x_n]$.

Now fix an orientation $[M]\in H_3(M,\Z)$; 
then the cup product on $M$ determines an alternating 
$3$-form $\mu_M$ on $H^1(M,\Z)$, given by 
$a\wedge b\wedge c\mapsto \langle a\cup b\cup c , [M]\rangle$. 
Let $\Pf(\mu_M)\in \Z[x_1,\dots, x_n]$ be the Pfaffian 
of $\mu_M$, as defined in \cite{Tu}.  As shown in \cite{Su-poinres},  
except for the trivial cases when $n\le 1$, the first resonance variety 
of $M$ is given by 
\begin{equation}
\label{eq:r13m-intro}
\RR^1_1(M)=
\begin{cases}
H^1(M,\C) & \text{if\/ $n$ is even},\\[2pt]
V(\Pf(\mu_M))  & \text{if\/ $n=2g+1\ge 3$ and $\mu_M$ is generic}.
\end{cases}
\end{equation}
Here, we say that $\mu_M$ is generic (in the sense of \cite{BP}) 
if there is an element 
$c\in H^1(M,\C)$ such that the $2$-form on $H^1(M,\C)$
given by $a \wedge b\mapsto \mu_A(a\wedge b\wedge c)$
has rank $2g$.

The higher depth resonance varieties also exhibit a nice pattern, 
revealed in \cite{Su-poinres}: 
$\RR^1_{2k}(M) =\RR^1_{2k+1}(M)$ if $n$ is even, and 
$\RR^1_{2k-1}(M) =\RR^1_{2k}(M)$  if $n$ is odd; 
moreover, if $n\ge 3$ and $\mu_M$ has maximal rank, 
then $\RR^1_{n-2}(M)=\RR^1_{n-1}(M)=\RR^1_{n}(M)=\{\zero\}$.
 
\subsection{A Tangent Cone theorem for closed $3$-manifolds}
\label{intro:tc3m}

As is well-known, $3$-manifolds may be non-formal, 
due to the presence of non-vanishing Massey products in their cohomology.  
Thus, we do not expect the Tangent Cone formula \eqref{eq:tcone-intro} 
to hold in this context. 

Nevertheless, something very special happens in degree $1$ and depth $1$.  
The next result  (proved in Theorem \ref{thm:tc 3d}), delineates exactly the 
class of  closed, orientable $3$-manifolds $M$ for which the second half of 
the Tangent Cone formula holds, except in the case when $n=b_1(M)$ is 
odd and at least $3$ and $\mu_M$ is {\em not}\/ generic, which remains open.

\begin{theorem}
\label{thm:tc3d-intro}
With notation as above,
\begin{enumerate}
\item \label{i-odd} 
If $n\le 1$, or $n$ is odd, $n\ge 3$, and $\mu_M$ 
is generic, then $\TC_{\one}(\VV^1_1(M))=\RR^1_1(M)$.
\item \label{i-even} 
If $n$ is even and $n\ge 2$, then $\TC_{\one}(\VV^1_1(M))= \RR^1_1(M)$ if 
and only if $\Delta_M= 0$.
\end{enumerate}
\end{theorem}

This result, together with those mentioned in \S\ref{intro:cjl}--\ref{intro:tc}, 
have  definite implications regarding the  kind of algebraic models 
a closed, orientable $3$-manifolds $M$ has, or, the kind of presentations 
the Malcev Lie algebra of $\pi=\pi_1(M)$ admits.
For instance, if $n$ is even, $n\ge 2$, and $\Delta_M\ne 0$, 
then $M$ is not $1$-formal and so $\m(\pi)$ admits no quadratic 
presentation. In Example \ref{ex:finite}, we exhibit a $3$-manifold $M$ 
with $b_1(M)=2$ for which the first half of the Tangent Cone formula fails, 
thus showing that $M$ actually has no $1$-finite $1$-model, or, 
equivalently, $\m(\pi)$ admits no finite presentation.

\subsection{Connected sums and graph manifolds}
\label{intro:ops}

As is well-known, every closed, orientable $3$-manifold decomposes 
as the connected sum of finitely many irreducible $3$-manifolds. 
We give in Theorem \ref{thm:jump conn} an explicit formula that 
expresses the cohomology jump loci of the connected sum 
of two closed, orientable, smooth $m$-manifolds ($m\ge 3$) 
in terms of the jump loci of the summands.  
Since every $3$-manifold is smooth, this reduces the 
computation of the cohomology jump loci of arbitrary 
closed, orientable $3$-manifolds to that of irreducible ones.
In particular, if $M=M_1\,\#\, M_2$, and both summands have 
non-zero first Betti number, then $\VV^1_1(M)=H^1(M,\C^*)$ 
and $\RR^1_1(M)=H^1(M,\C)$, and so the full Tangent Cone 
formula holds for $M$.

Every irreducible closed, orientable $3$-manifold $M$ 
admits a Jaco--Shalen--Johannson (JSJ) decomposition 
along incompressible tori; $M$ is a graph-manifold is each 
of the pieces is Seifert fibered. We discuss in \S\ref{sect:graphman} 
three classes of graph-manifolds where the cohomology jump loci 
can be described in a fairly detailed fashion:  
(1) closed, orientable Seifert manifolds with orientable base;  
(2) graph-manifolds whose closed-up Seifert pieces are 
of type (1) and whose underlying graph 
is a tree; and (3) boundary manifolds of complex 
projective line arrangements.

\subsection{Links in the $3$-sphere}
\label{intro:links}

In the final section we explore the extent to which 
the Tangent Cone formula applies to link complements.  Given 
a link $L=\{L_1,\dots, L_n\}$ in $S^3$, we let $X$ denote its 
complement. Then $\VV^1_1(X) = V(\Delta_L)\cup \{\one\}$, 
where $\Delta_L\in \Z[t_1^{\pm 1},\dots , t_n^{\pm 1}]$ is the 
(multivariable) Alexander polynomial of the link. Moreover, 
$\RR^1_1(X)$ is the vanishing locus of the codimension $1$ 
minors of the linearized Alexander matrix, whose entries 
are certain linear forms in the variables $x_1,\dots,x_n$, 
with coefficients solely depending on the linking numbers 
$\ell_{i,j}=\lk(L_i,L_j)$. 

For $2$-component links, we obtain a complete answer regarding 
the validity of the full Tangent Cone formula (in depth $1$), and 
the formality of the link complement.  In Theorem \ref{thm:2link} 
we show the following: the complement $X$ is formal if and only if 
$\tau_{\one}(\VV^1_1(X))=\TC_{\one}(\VV^1_1(X))=\RR^1_1(X)$, 
and this happens precisely when the linking number of the two 
components is non-zero.   We conclude with several examples 
of links with $3$ or more  components for which the second 
equality holds yet the first one does not, thereby showing that 
such link complements admit no $1$-finite $1$-models.

\section{Resonance varieties}
\label{sect:res}

\subsection{Commutative differential graded algebras}
\label{subsec:cdga}
Let $\k$ be a field of characteristic $0$, and 
let $A=(A^{\hdot},\D)$ be a commutative, differential 
graded algebra (for short, a $\cdga$) over $\k$.  
That is, $A$ is a non-negatively graded $\k$-vector 
space, endowed with a multiplication map 
$\cdot\colon A^i \otimes_\k A^j \to A^{i+j}$ satisfying 
$a\cdot b = (-1)^{ij} b \cdot a$,  
and a differential $\D\colon A^i\to A^{i+1}$ 
satisfying $\D(a\cdot b) = \D(a)\cdot b 
+(-1)^{i} a \cdot \D(b)$, for all $a\in A^i$ and $b\in A^j$. 
The cohomology of the underlying cochain complex, 
$H^{\hdot}(A)$, inherits the structure of a commutative, 
graded algebra ($\cga$); 
we will let $b_i(A)=\dim_\k H^i(A)$ be its Betti numbers.

A morphism between two $\cdga$s, $\varphi\colon A\to B$, is both 
an algebra map and a cochain map. Consequently, $\varphi$ induces a 
morphism $\varphi^*\colon H^{\hdot} (A)\to H^{\hdot} (B)$ 
between the respective cohomology algebras.  
We say that $\varphi$ is a quasi-isomorphism if $\varphi^*$ is an 
isomorphism. Likewise, we say $\varphi$ is a $q$-isomorphism (for some 
$q\ge 1$) if $\varphi^*$ is an isomorphism in degrees up to $q$ 
and a monomorphism in degree $q+1$.  

Two $\cdga$s $A$ and $B$ are {\em weakly equivalent} 
(or just {\em $q$-equivalent}) if there is a finite zig-zag of 
quasi-isomorphisms (or $q$-isomorphisms) connecting $A$ to $B$, 
\begin{equation}
\label{eq:zig-zag}
\xymatrixcolsep{20pt}
\xymatrix{
A\ar[r] & A_1 &A_2\ar[l] \ar[r]& \cdots & A_n\ar[r]\ar[l]  & B,
}
\end{equation}
with arrows going either way. 
In this case, we write $A\simeq B$ (or $A\simeq_q B$).  
A $\cdga$ $(A,\D)$ is said to be {\em formal}\/ (or just {\em $q$-formal}) 
if it is weakly equivalent (or just $q$-equivalent) to its cohomology 
algebra, $H^{\hdot}(A)$, endowed with the zero differential. 

\subsection{Resonance varieties}
\label{subsec:res}

Assume now that our $\cdga$ $A$ is connected, i.e., $A^0=\k$, 
generated by the unit $1$. Since $\D(1)=0$, we may 
identify the vector space $H^1(A)$ 
with $Z^1(A)=\ker (\D)$. For each element 
$a$ of this space, we turn $A$ into a cochain complex, 
\begin{equation}
\label{eq:aomoto}
\xymatrix{(A^{\hdot} , \delta_{a})\colon  \ 
A^0  \ar^(.65){\delta^0_{a}}[r] & A^1
\ar^(.5){\delta^1_{a}}[r] 
& A^2   \ar^(.5){\delta^2_{a}}[r]& \cdots },
\end{equation}
with differentials given by $\delta^i_{a} (u)= a \cdot u + \D(u)$, 
for all $u \in A^i$.  (The fact that $\delta^{i+1}_{a}\circ \delta^i_a=0$ 
for all $i\ge 0$  easily follows from the definitions.) Computing the 
homology of these chain complexes for various values of the 
parameter $a$, and keeping track of the dimensions of the resulting 
$\k$-vector spaces yields the sets
\begin{equation}
\label{eq:rra}
\RR^i_k(A)= \{a \in H^1(A)   
\mid  \dim_\k H^i(A^{\hdot}, \delta_{a}) \ge k\}.
\end{equation}

Suppose now that $A$ is {\em $q$-finite}, for some $q\ge 1$, 
that is, the Betti numbers $b_i=b_i(A)$ are finite for all $i\le q$.  
Clearly, $H^1(A)$ is also a finite-dimensional $\k$-vector space.  
Moreover, as we shall see in \S\ref{subsec:eqresvar}, the sets 
$\RR^i_k(A)$ are algebraic subsets 
of the ambient affine space $H^1(A)$, for all $i\le q$.  We call these  
sets the {\em resonance varieties}\/ of $A$, in degree 
$i\ge 0$ and depth $k\ge 0$.  
For each $0\le i \le q$, we obtain a descending filtration, 
\begin{equation}
\label{eq:res filt}
H^1(A)=\RR^i_0(A)\supseteq \RR^i_1(A)\supseteq \cdots \supseteq 
\RR^i_{b_i+1}(A)= \myempty .
\end{equation}
Clearly, $H^i(A^{\hdot}, \delta_{0})=H^i(A)$; thus, the point 
$\zero\in H^1(A)$ belongs to $\RR^i_k(A)$ 
if and only if $b_i\ge k$.  In particular, since $A$ is connected, 
we have that $\RR^0_1(A)=\{\zero\}$. 

In general, the resonance varieties of a $\cdga$ may not be invariant 
under scalar multiplication, see \cite{DP-ccm, MPPS, Su-tcone}. 
Nevertheless, when the differential of $A$ is zero (that is, 
$A$ is simply a $\cga$), the varieties $\RR^i_k(A)$ are 
homogeneous subsets of $H^1(A)=A^1$.  When $i=1$, 
these subsets admit a particularly simple 
description.  First note that the differential  
$\delta^0_a$ takes the generator $1\in A^0=\k$ to $a\in  A^1$. 
Thus, a non-zero element $a\in A^1$ belongs 
to $\RR^1_k(A)$ if and only if there exist elements 
$u_1,\dots ,u_k\in A^1$ such that the set $\{a,u_1,\dots ,u_k\}$ 
is linearly independent and $au_1=\cdots =au_k=0$ in $A^2$. 
In particular, if $b_1=0$ then $\RR^1_1(A)=\myempty$, and if 
$b_1=1$ then $\RR^1_1(A)=\{\zero\}$. 

\subsection{Fitting ideals }
\label{subsec:fitt}
Our next goal is to explain why the resonance varieties 
of a locally finite $\cdga$ are Zariski closed sets, and how to 
find defining equations for  these varieties. 
We start with some basic notions from commutative algebra, 
following Eisenbud \cite{Ei}. Let $S$ be a commutative ring with unit.   
If $\varphi$ is a matrix with entries in $S$, 
we let $I_k(\varphi)$ be the ideal of $S$ generated by all
minors of size $k$ of $\varphi$. We then have a descending 
chain of ideals, $S=I_0(\varphi)\supseteq I_1(\varphi)\supseteq \cdots$. 

Now suppose $S$ is Noetherian.  Then every finitely 
generated $S$-module $Q$ admits a finite presentation, say  
$S^p \xrightarrow{\,\varphi\,} S^q \to Q \to 0$.  We can 
arrange that $p\ge q$, by adding zero columns to the 
matrix $\varphi$ if necessary.   We then define 
the {\em $k$-th elementary ideal}\/ (or, {\em Fitting ideal}) 
of $Q$ as $E_k(Q)=I_{q-k}(\varphi)$.
As is well-known, this ideal depends only on the module $Q$, and 
not on the choice of presentation matrix $\varphi$, whence the notation.

The Fitting ideals  form an ascending chain,
 $E_0(Q)\subseteq E_1(Q)\subseteq \cdots \subseteq S$. 
Furthermore, $E_0(Q)\subseteq \Ann(Q)$ 
and $(\Ann (Q))^q \subseteq E_0(Q)$, 
while $\Ann (Q) \cdot E_k(Q)\subseteq E_{k-1}(Q)$, for 
all $k>0$.   Consequently, if we denote by $V(\mathfrak{a})\subset \Spec(S)$ 
the zero-locus of an ideal $\mathfrak{a}$, then $V(E_0(Q))=V(\Ann(Q))$. 

\subsection{Equations for the resonance varieties}
\label{subsec:eqresvar}
Once again, let $(A, \D)$ be a connected $\k$-$\cdga$ with $\dim_\k A^1<\infty$. 
Pick a basis $\{ e_1,\dots, e_n \}$ for the $\k$-vector space 
$H^1(A)$; to avoid trivialities, we shall assume that 
$n=b_1(A)$ is positive.  Let $\{ x_1,\dots, x_n \}$ be the Kronecker 
dual basis for the dual vector space $H_1(A)=(H^1(A))^*$.  
Upon identifying the symmetric algebra $\Sym(H_1(A))$ 
with the polynomial ring $S=\k[x_1,\dots, x_n]$, we 
obtain a cochain complex of finitely generated, free $S$-modules, 
\begin{equation}
\label{eq:univ aomoto}
\xymatrixcolsep{20pt}
(A\otimes S,\delta)\colon 
\xymatrix{
\cdots \ar[r] 
&A^{i}\otimes_{\k} S \ar^(.45){\delta^{i}_A}[r] 
&A^{i+1} \otimes_{\k} S \ar^(.5){\delta^{i+1}_A}[r] 
&A^{i+2} \otimes_{\k} S \ar[r] 
& \cdots},
\end{equation}
with differentials given by 
$\delta^{i}_A(u \otimes s)= \sum_{j=1}^{n} e_j u \otimes s x_j + \D u \otimes s$,  
for $u\in A^{i}$ and $s\in S$. 
It is readily verified that the evaluation of this cochain complex at an 
element $a\in H^1(A)$ coincides with the cochain complex $(A,\delta_a)$ 
from \eqref{eq:aomoto}.

Suppose that $A$ is $q$-finite, for some $q\ge 1$.  It is easy to see 
then that the sets $\RR^i_k(A)$ with $i<q$ are Zariski closed.
Indeed, an element $a \in A^1$ belongs to $\RR^i_k(A)$ if and only if 
\begin{equation}
\label{eq:rkdel}
\rank \delta^{i+1}_a + \rank \delta^{i}_a \le c_i -k.
\end{equation}
where $c_i=\dim_{\k} A^i$.  
Hence, $\RR^i_k(A)$ is the zero-set of the ideal generated 
by all minors of size $c_{i}-k+1$ of the block-matrix 
$\delta^{i+1}_A\oplus \delta^{i}_A$.  It turns out that the 
sets $\RR^q_k(A)$ are also Zariski closed even when 
$\dim_\k A^{q+1}=\infty$, see  \cite{DP-ccm, BW15}.

The degree $1$ resonance varieties  
$\RR^1_k(A)$ admit an even simpler description:  away 
from $\zero$, they are the vanishing loci of the codimension 
$k$ minors of $\delta^1_A$.  More precisely,
\begin{equation}
\label{eq:r1ka}
\RR^1_k(A)= \begin{cases}
V ( I_{n-k} (\delta^1_A) ) & \text{if $0<k<n$},\\ 
\{\zero\}  & \text{if $k=n$}.
\end{cases}
\end{equation}

\section{Characteristic varieties and the Alexander polynomial}
\label{sect:cvalex}

\subsection{Characteristic varieties}
\label{subsec:cv}
We say that a space $X$ is {\em $q$-finite}\/ 
(for some integer $q\ge 1$) if it has the homotopy type of a connected 
CW-complex with finite $q$-skeleton.  We will denote by 
$\pi=\pi_1(X,x_0)$ the fundamental group of such a space, 
based at a $0$-cell $x_0$.  Clearly, if the space $X$ is $1$-finite, 
the group $\pi$ is finitely generated, and if $X$ is $2$-finite, 
$\pi$ admits a finite presentation.  

So let $X$ be a $1$-finite space, and let 
$\Char(X)=\Hom(\pi,\C^*)$ be the group of complex-valued, 
multiplicative characters of $\pi$, whose identity $\one$ 
corresponds to the trivial representation. This is a complex 
algebraic group, which may be identified with $H^1(X,\C^*)$.  
The identity component, $\Char(X)^{0}$, is a an algebraic 
torus of dimension $n=b_1(X)$; the other connected components 
are translates of this torus by characters indexed by the torsion 
subgroup of $\pi_{\ab}=H_1(X,\Z)$. 

For each character $\rho\colon \pi\to \C^*$, let 
$\C_{\rho}$ be the corresponding rank $1$ local system on $X$.  
The {\em characteristic varieties}\/ of $X$ (in degree $i$ and depth $k$) 
are the jump loci for homology with coefficients in such local systems, 
\begin{equation}
\label{eq:cvx}
\VV^i_k(X)= \{\rho \in \Char(X) \mid  \dim H_i(X, \C_{\rho}) \ge k\}.
\end{equation}

In more detail, let $X^{\ab}\to X$ be the 
maximal abelian cover, with group of deck transformations 
$\pi_{\ab}$.  Upon lifting the cell structure of $X$ 
to this cover, we obtain a chain complex of $\Z[\pi_{\ab}]$-modules, 
\begin{equation}
\label{eq:equiv cc}
\xymatrixcolsep{24pt}
\xymatrix{\cdots \ar[r]& 
C_{i+1}(X^{\ab},\Z) \ar^(.53){\partial^{\ab}_{i+1}}[r] & 
 C_{i}(X^{\ab},\Z) \ar^(.45){\partial^{\ab}_{i}}[r] & 
  C_{i-1}(X^{\ab},\Z) \ar[r] & \cdots  
}.
\end{equation}

Tensoring this chain complex with the $\Z[\pi_{\ab}]$-module 
$\C_{\rho}$, we obtain a chain complex of $\C$-vector spaces,
\begin{equation}
\label{eq:eval cc}
\xymatrixcolsep{14pt}
\xymatrix{\cdots \ar[r]& C_{i+1}(X,\C_{\rho}) 
\ar^(.52){\partial^{\ab}_{i+1}(\rho)}[rr] &&
C_{i}(X,\C_{\rho}) \ar^(.47){\partial^{\ab}_{i}(\rho)}[rr] &&
C_{i-1}(X,\C_{\rho}) \ar[r] & \cdots  
},
\end{equation}
where the evaluation of $\partial^{\ab}_i$ at $\rho$ is obtained by applying 
the ring morphism $\C[\pi]\to \C$, $ g\mapsto \rho(g)$ to each entry.  
Taking homology in degree $i$ of this chain complex, we obtain the 
twisted homology groups $H_i(X, \C_{\rho})$ which appear in 
definition \eqref{eq:cvx}. 

If $X$ is a $q$-finite space, the sets $\VV^i_k(X)$ are Zariski closed 
subsets of the algebraic group $\Char(X)$, for all $i\le q$ and all $k\ge 0$.  
If $i<q$, this is again easy to see.  Indeed, let $R=\C[\pi_{\ab}]$ be the 
coordinate ring of the algebraic group $\Hom(\pi,\C^*)=\Hom(\pi_{\ab},\C^*)$. 
By definition, a character $\rho\in \Char(X)$ belongs to $\VV^i_k(X)$ 
if and only if
\begin{equation}
\label{eq:rho rank}
\rank \partial^{\ab}_{i+1}(\rho) + 
\rank \partial^{\ab}_{i}(\rho) \le c_i -k,
\end{equation}
where $c_i=c_i(X)$ is the number of $i$-cells of $X$.  
Hence, $\VV^i_k(X)$ is the zero-set of the ideal of 
minors of size $c_i-k+1$ of the block-matrix 
$\partial^{\ab}_{i+1} \oplus \partial^{\ab}_{i}$.  The case 
$i=q$ is covered in \cite[Lemma 2.1]{PS-plms} and 
\cite[Proposition 4.1]{PS-mrl}.

Clearly, $\VV^i_0(X)=\Char(X)$.  Moreover, 
$\one\in \VV^i_k(X)$ if and only if the 
$i$-th Betti number $b_i(X)$ is at least $k$. In degree $0$, 
we have that $\VV^0_1(X)= \{ \one \}$ and $\VV^0_k(X)= \myempty$ 
for $k>1$. 
In degree $1$, the sets $\VV^1_k(X)$ depend only 
on the fundamental group $\pi=\pi_1(X,x_0)$, and, in fact, only on 
its maximal metabelian quotient, $\pi/\pi''$; thus, we shall sometimes 
write these sets as $\VV^1_k(\pi)\subseteq\Char(\pi)$, 
and refer to them as the characteristic varieties of $\pi$.  

If $b_1(\pi)=0$, then $\Char(\pi)\subset \C^*$ is a finite set of roots of unity 
in bijection with $\pi_{\ab}$; although $1\notin \VV^1_1(\pi)$, 
other roots of unity may belong to $\VV^1_1(\pi)$.  For instance, 
if $\pi=\Z_2$, then $\Char(\pi)=\{1,-1\}$, while  $\VV^1_1(\pi)=\{-1\}$.

\subsection{Alexander varieties}
\label{subsec:alexinv}

There is an alternative, very useful interpretation of the de\-gree one  
characteristic varieties, first noted by Hironaka in \cite{Hi97}.  Namely, 
let $B_X= H_1(X^{\ab},\Z)$ be the {\em Alexander invariant}\/ of a $2$-finite 
space $X$, 
viewed as a $\Z[\pi_{\ab}]$-module, and let $\wV^1_k(X)=V(E_{k-1}(B_X \otimes \C))$ 
be the zero sets of the elementary ideals of the complexification 
of this module.  Then, at least away from the trivial representation, 
the degree $1$ characteristic varieties of $X$ coincide with the 
{\em Alexander varieties},
\begin{equation}
\label{eq:ve}
\VV^1_k(X) \setminus \set{\one} = \wV^1_k(X)  \setminus \set{\one}.
\end{equation}

Indeed, if $\rho\colon \pi\to \C^*$ is a non-trivial character, 
then, by the universal coefficients theorem, 
$H_1(X, \C_{\rho})$ has dimension at least $k$ if and only if 
$(B_X \otimes \C)\otimes_{\C[\pi_{\ab}]} \C_{\rho}$ has dimension 
at least $k$; in turn, this condition is equivalent to $\rho \in V(E_{k-1}(B_X \otimes \C))$. 

More generally, one may define the Alexander varieties  $\wV^i_k(X)$
as the zero sets of the ideals  $E_{k-1}(H_i(X^{\ab},\C))$.  Provided 
$X$ is $q$-finite, a formula analogous to \eqref{eq:ve} holds for all 
$i\le q$, but only in depth $k=1$, see  \cite[Corollary 3.7]{PS-plms}.

If $X$ is $2$-finite, the degree $1$ characteristic varieties can be 
computed algorithmically, starting from a finite presentation of the 
group $\pi=\pi_1(X)$. If $\pi=\langle  x_1,\dots ,x_{m}\mid r_1,\dots ,r_{s}\rangle$
is such a presentation, then $\partial^{\ab}_2\colon \Z[\pi_{\ab}]^s \to \Z[\pi_{\ab}]^m$, 
the second boundary map in the chain complex  \eqref{eq:equiv cc}, coincides with the 
Alexander matrix $\big(\partial_j r_i \big)^{\ab}$ of abelianized Fox derivatives of the 
relators.  An argument as above shows that $\VV^1_k(\pi)$ coincides, 
at least away from $\one$, with the zero locus of the ideal of codimension $k$ 
minors of $\partial^{\ab}_2$; that is, 
\begin{equation}
\label{eq:v1s}
\VV^1_k(\pi) \setminus \set{\one} =
V(E_{k}(\coker \partial^{\ab}_2)) \setminus \set{\one} .
\end{equation}

The characteristic varieties of a space or a group can be arbitrarily complicated.  
For instance, let $f\in \Z[t_1^{\pm 1},\dots , t_n^{\pm 1}]$ be an integral Laurent polynomial. 
Then, as shown in \cite{SYZ}, there is a finitely presented group $\pi$ 
with $\pi_{\ab}=\Z^n$ such that $\VV^1_1(\pi)=V(f)  \cup \{\one\}$.  
More generally, let $Z$ be a an algebraic subset of $(\C^*)^n$, defined 
over $\Z$, and let $j$ be a positive integer.  Then, as shown in 
\cite{Wa}, there is a finite, connected CW-complex $X$ with 
$\Char(X)=(\C^*)^n$ such that $\VV^i_1(X)=\{\one\}$ for $i<j$ and 
$\VV^j_1(X)=Z  \cup \{\one\}$. 

\subsection{The Alexander polynomials of a space}
\label{subsec:alexpoly}
Let $X$ be a $2$-finite space, with fundamental group $\pi=\pi_1(X)$.  
We shall let $H=\pi_{\ab}/\tor(\pi_{\ab})$ be the maximal torsion-free 
abelian quotient of $\pi$.  It is readily seen that the group ring $\Z[H]$ is a 
commutative Noetherian ring and a unique factorization domain.  

Let $q\colon X^{H}\to X$ be the regular cover corresponding to the 
projection $\pi\surj H$, i.e., the maximal torsion-free abelian cover of $X$.  
Fixing a basepoint $x_0\in X$, the {\em Alexander module}\/ of $X$ is 
defined as the relative homology group $A_X= H_1(X^H,q^{-1}(x_0), \Z)$, 
viewed as a $\Z[H]$-module.
For each integer $k\ge 0$, the {\em $k$-th Alexander ideal}\/ 
is the determinantal ideal $E_k(A_X)$, while the {\em $k$-th
Alexander polynomial}\/ is $\Delta_X^k=\gcd ( E_k(A_X))$, the 
greatest common divisor of the elements in the ideal 
$E_k(A_X)\subseteq \Z[H]$.

Fixing a basis for $H\cong \Z^n$, we may identify the group ring $\Z[H]$ 
with the ring of Laurent polynomials in $t_1^{\pm1},\dots , t_n^{\pm 1}$. 
The Laurent polynomials $\Delta_{X}^k\in \Z[H]$ are 
well-defined up to multiplication by units in this ring, i.e., 
 monomials of the form $\pm t_1^{a_1}\cdots t_n^{a_n}$ 
(the equivalence relation is written as $\doteq$).

Of particular importance is the polynomial $\Delta_X:=\Delta_X^0$,
simply called the {\em Alexander polynomial}\/ of $X$. As 
shown in \cite[Lemma II.5.5]{Tu}, if $b_1(X)\ge 2$ 
this ideal is contained in $\Delta_{X}\cdot I_H$, 
where $I_H=\ker(\varepsilon\colon \Z[H]\to \Z)$ is the 
augmentation ideal. 

\subsection{The zero sets of the Alexander polynomials}
\label{subsec:alexchar}

Henceforth, we identify the identity component of the character torus, 
$\Char(X)^0$, with the algebraic torus $\Hom(H,\C^*)=(\C^*)^n$, 
where recall $H$ is the torsion-free part of $H_1(X,\Z)$ and $n=b_1(X)$.
The Laurent polynomials in $n$ variables are precisely 
the regular functions on this algebraic torus. 
As such, the Alexander polynomials $\Delta_{X}^k$ 
define algebraic hypersurfaces, 
\begin{equation}
\label{eq:vdelta}
V(\Delta^k_X) = \set{\rho \in \Char(X)^0\mid \Delta^k_X(\rho)=0}.
\end{equation}

Write $\ZZ^1_k(X)=\VV^1_k(X)\cap \Char(X)^0$, and let 
$\cz^1_k(X)$ be the union of all
codimension-one irreducible components of $\ZZ^1_k(X)$. 
The next lemma details the relationships between the
hypersurfaces defined by the Alexander polynomials of
$X$ and the degree $1$ characteristic varieties of $X$. 

\begin{lemma}[\cite{DPS-imrn, FS14}]
\label{lem:delta cv}
For each $k\ge 1$, the following hold.
\begin{enumerate}
\item \label{dc1}
The polynomial $\Delta^{k-1}_X$ is identically $0$ if and only if
$\ZZ^1_k(X) = \Char(X)^0$, in which case $\cz^1_k(X)=\myempty$.
\item \label{dc2}
Suppose that $b_1(X)\ne 0$ and $\Delta^{k-1}_X\ne 0$. Then
$\cz^1_k(X)=V(\Delta^{k-1}_X)$ if $b_1(X)\ge 2$, 
and $\cz^1_k(X)=V(\Delta^{k-1}_X)\sqcup \{ \one\}$ if $b_1(X)=1$.
\item \label{dc3}
Suppose that $b_1(X)\ge 2$.  Then $\Delta^{k-1}_X\doteq 1$ if and only 
if $\cz^1_k(X)=\myempty$.
\end{enumerate}
\end{lemma}

In particular, if $b_1(X)=0$, then $\Delta_X= 0$ and 
$\ZZ^1_1(X) = \Char(X)^0 = \{1\}$.

In a special type of situation (singled out in \cite{DPS-imrn}), 
the relationship between the first characteristic variety and the Alexander 
polynomial is even tighter.   

\begin{proposition}
\label{prop:zz1-pi}
Suppose $I^s_H\cdot ( \Delta_{X} ) \subseteq E_1(A_X)$, 
for some $s\ge 0$. Then 
\[
\ZZ^1_1(X) = V(\Delta_{X}) \cup \set{\one}.
\]
\end{proposition}

In particular, if $H_1(X,\Z)$ is torsion-free and the assumption of 
Proposition \ref{prop:zz1-pi} holds, then $\VV^1_1 (X)$ itself 
coincides with $V(\Delta_X)$, at least away from $\one$; if, 
moreover, $\Delta_X(\one)=0$, then $\VV^1_1 (X)=V(\Delta_X)$.

\section{Algebraic models and the tangent cone theorem}
\label{sect:tcone}

\subsection{Tangent cones}
\label{subsec:exp tc}

We start by reviewing two constructions which provide 
approximations to a subvariety $W$ of a complex algebraic 
torus $(\C^*)^n$. The first one is the classical tangent cone, 
while the second one is the exponential tangent cone, 
a construction introduced in \cite{DPS-duke} and further 
studied in \cite{Su-imrn,DP-ccm,SYZ}.

Let $I$ be an ideal in the Laurent polynomial ring   
$\C[t_1^{\pm 1},\dots , t_n^{\pm 1}]$ such that $W=V(I)$.   
Picking a finite generating set for $I$, and multiplying 
these generators with suitable monomials if necessary, 
we see that $W$ may also be defined by the ideal $I\cap R$ 
in the polynomial ring $R=\C[t_1,\dots,t_n]$.  
Let $J$ be the ideal  in the polynomial ring 
$S=\C[x_1,\dots, x_n]$ generated by the polynomials 
$g(x_1,\dots, x_n)=f(x_1+1, \dots , x_n+1)$, 
for all $f\in I\cap R$. 

The {\em tangent cone}\/ of $W$ at $\one \in (\C^*)^n$ is the algebraic 
subset $\TC_{\one}(W)\subseteq \C^n$ defined by the ideal 
$\init(J)\subset S$ generated by the initial forms of 
all non-zero elements from $J$.  The set 
$\TC_{\one}(W)$ is a homogeneous subvariety of $\C^n$, 
which depends only on the analytic germ of $W$ at 
the identity.  In particular, $\TC_{\one}(W)\ne \myempty$ 
if and only if $\one\in W$.  

Let $\exp\colon \C^n \to (\C^*)^n$ be the exponential map, 
given in coordinates by $x_i\mapsto e^{x_i}$.  
The {\em exponential tangent cone}\/ at $\one$ 
to a subvariety $W\subseteq (\C^*)^n$ is the set  
\begin{equation}
\label{eq:tau1}
\tau_{\one}(W)= \{ x\in \C^n \mid \exp(\lambda x)\in W,\ 
\text{for all $\lambda \in \C$} \}.
\end{equation}
It is readily seen that $\tau_{\one}$ commutes with finite unions and 
arbitrary intersections. Furthermore, $\tau_{\one}(W)$ only depends 
on $W_{(\one)}$, the analytic germ of $W$ at the identity; in particular, 
$\tau_{\one}(W)\ne \myempty$ if and only if $\one\in W$.  The main 
property of this construction is encapsulated in the following lemma.

\begin{lemma}[\cite{DPS-duke, Su-imrn, SYZ}] 
\label{lem:exp-tcone}
The exponential tangent cone $\tau_{\one}(W)$  of a subvariety 
$W\subseteq (\C^*)^n$ is a finite union of rationally defined linear 
subspaces of the affine space $\C^n$.  
\end{lemma}
 
For instance, if $W$ is an algebraic subtorus of $(\C^{*})^n$,  
then $\tau_{\one}(W)$ equals $\TC_{\one}(W)$, and both 
coincide with $T_{\one}(W)$, 
the tangent space to $W$ at the identity $\one$. 
More generally, there is always an inclusion between the two 
types of tangent cones associated to an  algebraic subset 
$W\subseteq (\C^{*})^n$, namely, 
\begin{equation}
\label{eq:ttinc}
\tau_{\one}(W)\subseteq \TC_{\one}(W). 
\end{equation}

As we shall see,  though, 
this inclusion is far from being an equality for arbitrary $W$. 
For instance, the tangent cone $\TC_{\one}(W)$ may 
be a non-linear, irreducible subvariety of $\C^n$, or $\TC_{\one}(W)$ 
may be a linear space containing the exponential tangent cone 
$\tau_{\one}(W)$ as a union of proper linear subspaces.

\subsection{The Exponential Ax--Lindemann theorem}
\label{subsec:ax}

In \cite{BW17}, Budur and Wang establish the following version of the 
Exponential Ax--Lindemann theorem \cite{Ax}, which proves 
to be very useful in this context.

\begin{theorem}[\cite{BW17}]
\label{thm:ax}
Let $V\subseteq \C^n$ and $W\subseteq (\C^*)^n$ be irreducible algebraic subvarieties.
\begin{enumerate}
\item \label{bw1}
Suppose $\dim V=\dim W$ and $\exp(V)\subseteq W$.  Then $V$ is a translate 
of a linear subspace, and $W$ is a translate of an algebraic subtorus. 
\item \label{bw2}
Suppose the exponential map $\exp\colon \C^n \to (\C^*)^n$ induces a 
local analytic isomorphism  $V_{(\zero)} \to W_{(\one)}$. Then  
$W_{(\one)}$ is the germ of an algebraic subtorus. 
\end{enumerate}
\end{theorem}

A standard dimension argument shows the following: 
if $W$ and $W'$ are irreducible algebraic subvarieties of 
$(\C^*)^n$ which contain $\one$ and whose germs at 
$\one$ are locally analytically isomorphic, then $W\cong W'$.  
Using this fact, we obtain the following corollary to part \eqref{bw2} of 
the above theorem.

\begin{corollary}
\label{cor:ax-bis}
Let $V\subseteq \C^n$ and $W\subseteq (\C^*)^n$ be irreducible 
algebraic subvarieties.  Suppose the exponential map 
$\exp\colon \C^n \to (\C^*)^n$ induces a local analytic 
isomorphism $V_{(\zero)} \cong W_{(\one)}$. Then $W$ is 
an algebraic subtorus and $V$ is a rationally defined 
linear subspace.
\end{corollary}

\subsection{Tangent cones and jump loci}
\label{subsec:tcone thm}

Let $X$ be a $q$-finite space. Its cohomology algebra, 
$H^{\hdot}(X,\C)$, is then $q$-finite; thus, the resonance varieties 
$\RR^i_k(X):=\RR^i_k(H^{\hdot}(X,\C))$ are homogeneous 
algebraic subsets of the affine space $H^{1}(X,\C)$, for all $i\le q$ 
and $k\ge 0$.  

The following basic relationship between the characteristic 
and resonance varieties was established by Libgober in \cite{Li02} in 
the case when $X$ is a finite CW-complex and $i$ is arbitrary; a similar 
proof works in the generality that we work in here (see \cite{Su-tcone, DS18} 
for an even more general setup).

\begin{theorem}[\cite{Li02}]
\label{thm:lib}
Suppose $X$ is a $q$-finite space.  Then, for all $i\le q$ and $k\ge 0$, 
\begin{equation}
\label{eq:tc lib}
\TC_{\one}(\VV^i_k(X))\subseteq \RR^i_k(X).
\end{equation}
\end{theorem}

Putting together these inclusions with those from \eqref{eq:ttinc}, 
we obtain the following corollary. 
\begin{corollary}
\label{cor:tcone inc}
Suppose $X$ is a $q$-finite space.  Then, for all $i\le q$ and $k\ge 0$, 
\begin{equation}
\label{eq:tc inc}
\tau_{\one}(\VV^i_k(X))\subseteq  
\TC_{\one}(\VV^i_k(X))\subseteq \RR^i_k(X).
\end{equation}
\end{corollary}
Note that we may replace in Corollary \ref{cor:tcone inc} the characteristic varieties 
$\VV^i_k(X)$ by the subvarieties $\ZZ^i_k(X)=\VV^i_k(X)\cap \Char^0(X)$.  
Also note that, if $\RR^i_k(X)$ is empty or equal to $\{\zero\}$, then all of the 
above inclusions become equalities. In particular, $\tau_{\one}(\VV^1_1(X))=
\TC_{\one}(\VV^1_1(X))= \RR^1_1(X)$ if $b_1(X)\le 1$. 
In general, though, each of the inclusions from \eqref{eq:tc inc}---%
or both---can be strict, as examples to follow will show. 

A particular case of the above corollary is worth mentioning separately.

\begin{corollary}
\label{cor:tcone gps}
Let $\pi$ be a finitely generated group.  Then, for all $k\ge 0$, 
\[
\tau_{\one}(\VV^1_k(\pi))\subseteq  
\TC_{\one}(\VV^1_k(\pi))\subseteq \RR^1_k(\pi).
\]
\end{corollary}

\subsection{Algebraic models for spaces and groups}
\label{subsec:algmod}

Given a space $X$, we let $A_{\PL}^\hdot(X)$ be the 
commutative differential graded $\Q$-algebra of 
rational polynomial forms, as defined by Sullivan in \cite{Su77} 
(see \cite{FHT, FHT2, FOT} for a detailed exposition).  
There is then a natural isomorphism 
$H^\hdot(A_{\PL}(X)) \cong H^\hdot(X,\Q)$ under which the respective 
induced homomorphisms in cohomology correspond. 

As before, let $\k$ be a field of characteristic $0$, and $q$ a positive integer.  
We say that a $\k$-$\cdga$ $(A,d)$ is a {\em model}\/ (or just a {\em $q$-model}) over $\k$ 
for $X$ if $A$ is weakly equivalent (or just $q$-equivalent) to $A_{\PL}(X)\otimes_{\Q} \k$. 
For instance, if $X$ is a smooth manifold, then $\Omega^{\hdot}_{\dR}(X)$,  
the de Rham algebra of smooth forms on $X$, is a model of $X$ over $\R$.  
By considering a classifying space $X=K(\pi, 1)$ for a group $\pi$, we may speak 
about $q$-models for groups. 

A continuous map $f\colon X\to Y$ is said to be a {\em $q$-rational 
homotopy equivalence}\/ if the induced homomorphism 
$f^*\colon H^{i}(Y,\Q)\to H^{i}(X,\Q)$ is an isomorphism for 
$i\le q$ and a monomorphism for $i=q+1$.  
Such a map induces a $q$-equivalence 
$A_{\PL}(f)\colon A_{\PL}(Y) \to A_{\PL}(X)$. Therefore,   
whether a space $X$ admits a $q$-finite $q$-model depends 
only on its $q$-rational homotopy type, in particular, on its 
$q$-homotopy type.   Consequently,  a path-connected  space $X$ 
admits a $1$-finite $1$-model if and only if the fundamental 
group $\pi=\pi_1(X)$ admits one.  The existence of such a model 
puts rather stringent constraints on the group $\pi$.  One such 
constraint is given in \cite[Theorem~1.5]{PS-jlms}.

Following Quillen \cite{Qu}, let us define the Malcev Lie algebra of $\pi$, 
denoted $\m(\pi)$, as the complete, filtered Lie algebra of primitive 
elements in the $I$-adic completion of  the Hopf algebra 
$\Q[\pi]$, where $I=\ker(\varepsilon \colon \Q[\pi] \to \Q)$ 
is the augmentation ideal.   That is to say, 
$\m(\pi)=\operatorname{Prim} (\widehat{\Q[\pi]})$, where 
$\widehat{\Q[\pi]}=\varprojlim_{r} \Q[\pi]/I^r$.

\begin{theorem}[\cite{PS-jlms}]
\label{thm:psobs}
A finitely generated group $\pi$ admits a $1$-finite $1$-model if and 
only if the Malcev Lie algebra $\m(\pi)$ is the lower central series (LCS) 
completion of a finitely presented Lie algebra.
\end{theorem}

The above condition means that $\m(\pi)=\widehat{L}$, for some finitely presented 
Lie algebra $L$, where $\widehat{L}=\varprojlim_{r} L/\Gamma_r L$, with 
the LCS series $\{\Gamma_r L\}_{r\ge 1}$ defined inductively by 
$\Gamma_1 L=L$ and $\Gamma_r L=[L,\Gamma_{r-1} L]$ for $r>1$. 

\subsection{Algebraic models and cohomology jump loci}
\label{subsec:mod-cjl}

Work of Dimca and Papadima \cite{DP-ccm}, 
generalizing previous work from \cite{DPS-duke}, establishes a 
tight connection between the geometry of the characteristic 
varieties of a space and that of resonance varieties of 
a model for it, around the origins of the respective ambient spaces, 
provided certain finiteness conditions hold.

More precisely, let $X$ be a path-connected space with $b_1(X)<\infty$, 
and consider the analytic map $\exp\colon H^1(X,\C)\to H^1(X,\C^*)$ 
induced by the coefficient homomorphism $\C\to \C^*$, $z\mapsto e^z$.
Let $(A,d)$ be a $\cdga$ model for $X$, defined over $\C$.  
Upon identifying $H^1(A)\cong H^1(X,\C)$, we obtain an analytic map 
$H^1(A)\to H^1(X,\C^*)$, which takes $\zero$ to $\one$.   

\begin{theorem}[\cite{DP-ccm}]
\label{thm:thmb}
Let $X$ be a $q$-finite space, and suppose $X$ admits a $q$-finite, 
$q$-model $A$, for some $q\ge 1$. Then, the aforementioned map,  
$H^1(A)\to H^1(X,\C^*)$, induces a local analytic isomorphism 
$H^1(A)_{\zero} \to H^1(X,\C^*)_{\one}$, 
which identifies the germ at $\zero$ of $\RR^i_k(A)$ 
with the germ at $\one$ of $\VV^i_k(X)$, for all $i\le q$ and all $k\ge 0$. 
\end{theorem}

Recent work of Budur and Wang \cite{BW17} builds on this theorem, 
providing a structural result on the geometry of the characteristic varieties 
of spaces satisfying the hypothesis of the above theorem.   
Putting together Theorem \ref{thm:thmb} and Corollary \ref{cor:ax-bis} 
yields their result, in the slightly stronger form given in \cite{PS-jlms}.

\begin{theorem}[\cite{BW17}]
\label{thm:bw}
Suppose $X$ is a $q$-finite space which admits a $q$-finite $q$-model. 
Then all the irreducible components of 
$\VV^i_k(X)$ passing through $\one$ are algebraic subtori  of 
$H^1(X,\C^*)$, for all $i\le q$ and $k\ge 0$.
\end{theorem}

As an immediate corollary of the previous two theorems, we 
obtain the following ``Tangent Cone formula."

\begin{theorem}
\label{thm:tcone-fm}
Suppose $X$ is a $q$-finite space which admits a $q$-finite $q$-model $A$.  
Then, for all $i\le q$ and 
$k\ge 0$,
\begin{equation}
\label{eq:tc-fm}
\tau_{\one}(\VV^i_k(X))=\TC_{\one}(\VV^i_k(X))=\RR^i_k(A).
\end{equation}
\end{theorem}

This theorem, together with Theorem \ref{thm:psobs}, yields the 
following corollary.

\begin{corollary}
\label{cor:malcone}
Suppose $\pi$ is a finitely generated group whose Malcev Lie algebra   
is the LCS completion of a finitely presented Lie algebra. Then 
$\tau_{\one}(\VV^1_k(\pi))=\TC_{\one}(\VV^1_k(\pi))$, for all $k\ge 0$.
\end{corollary}

In other words, if the first half of the Tangent Cone formula fails in degree $1$, 
i.e., if $\tau_{\one}(\VV^1_k(\pi))\subsetneqq \TC_{\one}(\VV^1_k(\pi))$ for 
some $k> 0$, then $\m(\pi)\not\cong \widehat{L}$, for any finitely presented 
Lie algebra $L$.  This will happen automatically if the variety 
$\TC_{\one}(\VV^1_k(\pi))$ has an irreducible component which is not 
a rationally defined linear subspace of $H^1(\pi,\C)$.

\subsection{Formality}
\label{subsec:formal}

A path-connected space $X$ is said to be {\em formal}\/ (over a field $\k$ 
of characteristic $0$) 
if Sullivan's algebra $A_{\PL}(X)\otimes_{\Q} \k$ is formal; in other words, 
if the cohomology algebra, $H^{\hdot}(X,\k)$, endowed with the zero 
differential, is weakly equivalent to $A_{\PL}(X)\otimes_{\Q} \k$. Likewise, 
a space $X$ is merely {\em $q$-formal}\/  (for some $q\ge 1$) if 
$A_{\PL}(X)\otimes_{\Q} \k$ has this property.  These formality and partial 
formality notions are independent of the field $\k$, as long as its 
characteristic is $0$. Furthermore, if $X$ is a $q$-formal CW-complex 
of dimension at most $q+1$, then $X$ is formal, cf.~\cite{Mc10}.

Evidently, every $q$-finite, $q$-formal space $X$ admits a $q$-finite $q$-model, 
namely, $A=H^{\hdot}(X,\k)$ with $d=0$. 
Examples of formal spaces include suspensions, rational cohomology 
tori, surfaces, compact connected Lie groups, as well 
as their classifying spaces.   On the other hand, the 
only nilmanifolds which are formal are tori. Formality 
is preserved under wedges and products of spaces, 
and connected sums of manifolds.  

It is readily seen that the $1$-formality property of a space 
$X$ depends only on its fundamental group, $\pi=\pi_1(X)$.  
Alternatively, a finitely generated group $\pi$ is $1$-formal 
if and only if its Malcev Lie algebra $\m(\pi)$ is isomorphic 
to the LCS completion of a finitely generated, quadratic 
Lie algebra  $L$.  Examples of $1$-formal groups include free 
groups and free abelian groups of finite rank, surface groups, 
and groups with first Betti number equal to $0$ or $1$.  
The $1$-formality property is preserved under finite free 
products and direct products of (finitely generated) groups.  
We refer to \cite{Mc10, PS-formal, PS-jlms, SW} for more 
details and references regarding all these notions.

\subsection{Formality and cohomology jump loci}
\label{subsec:formal-cjl}

The main connection between the formality property 
of a space and the geometry of its cohomology jump 
loci is provided by the next result. This result, which 
was first proved in degree $i=1$ in \cite{DPS-duke}, 
and in arbitrary degree in \cite{DP-ccm}, is now an 
immediate consequence of Theorem~\ref{thm:tcone-fm}.
 
\begin{corollary}
\label{cor:tcone}
Let $X$ be a $q$-finite, $q$-formal space. Then, for all $i\le q$ and 
$k\ge 0$,
\begin{equation}
\label{eq:tc}
\tau_{\one}(\VV^i_k(X))=\TC_{\one}(\VV^i_k(X))=\RR^i_k(X).
\end{equation}
\end{corollary}

In particular, if $\pi$ is a finitely generated, $1$-formal group, then, for all $k\ge 0$, 
\begin{equation}
\label{eq:tc-pi}
\tau_{\one}(\VV^1_k(\pi))= 
\TC_{\one}(\VV^1_k(\pi))=\RR^1_k(\pi).
\end{equation}

As an application of Corollary \ref{cor:tcone}, we have the following characterization 
of the irreducible components of the cohomology jump loci in the formal setting.

\begin{corollary}
\label{cor:rational}
Suppose $X$ is a $q$-finite, $q$-formal space.  Then, for 
all $i\le q$ and $k\ge 0$, the following hold.
\begin{enumerate}
\item \label{tc1}
All irreducible components of the resonance varieties  
$\RR^i_k(X)$ are rationally defined linear subspaces 
of $H^1(X,\C)$. 
\item \label{tc2}
All irreducible components of the characteristic varieties  
$\VV^i_k(X)$ which contain the origin are algebraic 
subtori of $\Char(X)^{0}$, of the form $\exp(L)$, where 
$L$ runs through the linear subspaces comprising $\RR^i_k(X)$.
\end{enumerate}
\end{corollary}

\section{Resonance varieties of $3$-manifolds}
\label{sect:res man}

We now switch our focus from the general theory of cohomology 
jump loci to some of the applications of this theory in low-dimensional 
topology. We start by describing the resonance varieties attached 
to the cohomology ring of a closed, orientable, $3$-dimensional 
manifold, based on the approach from \cite{Su-poinres}.

\subsection{The intersection form of a $3$-manifold}
\label{subsec:coho 3-mfd}

Let $M$ be a compact, connected $3$-mani\-fold without boundary. 
For short, we shall refer to $M$ as being a {\em closed}\/ $3$-manifold.
Throughout, we will also assume that $M$ is orientable.   

Fix an orientation class $[M]\in H_3(M,\Z)\cong\Z$. With 
this choice, the cup product on $M$ determines an alternating 
$3$-form $\mu_M$ on $H^1(M,\Z)$, given by 
\begin{equation}
\label{eq:eta}
\mu_M(a\wedge b\wedge c) = \langle a\cup b\cup c , [M]\rangle,
\end{equation} 
where $\langle \cdot, \cdot \rangle$ denotes the Kronecker pairing. 
In turn, the cup-product map $\bigwedge^2 H^1(M,\Z) 
\to H^2(M,\Z)$ is determined by the intersection form $\mu_M$ 
via $\langle a\cup b , \gamma \rangle = \mu_M (a\wedge b \wedge c)$, 
where $c$ is the Poincar\'{e} dual of 
$\gamma \in H_2(M,\Z)$.  

Now fix a basis $\{e_1,\dots ,e_n\}$ for $H^1(M,\Z)$, 
and choose  $\{e^{\vee}_1,\dots ,e^{\vee}_n\}$ as basis 
for the torsion-free part of $H^2(M,\Z)$, where $e^{\vee}_i$ 
denotes the Kronecker dual of the Poincar\'e dual of $e_i$. 
Write 
\begin{equation}
\label{eq:muijk}
\mu_M=\sum_{1\le i<j<k\le n} \mu_{ijk} e_i  e_j  e_k,
\end{equation}
where $\mu_{ijk}=\mu(e_i\wedge e_j\wedge e_k )$.  Using 
formula \eqref{eq:eta}, we find that 
$e_i  e_j=\sum_{k=1}^{n} \mu_{ijk}  e^{\vee}_k$.  

As shown by Sullivan \cite{Su75}, for every finitely generated, 
torsion-free abelian group $H$ and every $3$-form $\mu\in \bwedge^3 H^*$, 
there is a closed, oriented $3$-manifold $M$ with $H^1(M,\Z)=H$ 
and cup-product form $\mu_M=\mu$. Such a $3$-manifold can 
be constructed by a process known as ``Borromean surgery" 
or ``$T^3$-surgery", see for instance 
\cite[Corollary~3.5]{CGO} or \cite[Theorem~6.1]{My11}. 
More precisely, if $n=\rank H$, such a manifold $M$ may be defined as 
$0$-framed surgery on a link in $S^3$ obtained from the trivial 
$n$-component link by replacing a collection of trivial $3$-string braids by 
a collection of $3$-string braid whose closure is the Borromean rings.

Of course, there are many closed $3$-manifolds that realize a given intersection 
$3$-form $\mu$. For instance, if $M$ is such a manifold, then the connected sum 
of $M$ with any rational homology $3$-sphere will also realize $\mu$.  
As another example, if $M$ is the link in $S^5$ of an isolated singularity 
of a complex algebraic surface, then $\mu_M=0$ 
\cite{Su75}; more generally, if $M$ bounds a compact, orientable $4$-manifold 
$W$ such that the cup-product pairing on  $H^2(W,M)$ is non-degenerate, 
then $\mu_M=0$, see \cite[Proposition~13]{Ma08}.

\begin{remark}
\label{rem:homcob}
Two closed, oriented $3$-manifolds, $M_1$ and $M_2$, are said to be 
{\em homology cobordant}\/ if there is a compact oriented $4$-manifold $W$ 
such that  $\partial W=M_1\sqcup - M_2$ and the inclusion-induced maps 
$H_n(M_i,\Z)\to H_n(W,\Z)$ are isomorphisms for $i=1,2$ and all $n$. It 
is readily seen that homology cobordism is an equivalence relation.  Moreover, 
if $M_1$ and $M_2$ are homology cobordant, then their cohomology rings 
are isomorphic. 
\end{remark}

\subsection{Resonance varieties of $3$-manifolds}
\label{subsec:res3d}

Let $A$ be a graded, graded-commutative algebra over 
a field $\k$ such that $A$ is connected and all the Betti numbers 
$b_i(A)=\dim_{\k} A^i$ are finite. We say that $A$ is a {\em Poincar\'{e} duality 
algebra}\/ of dimension $m$ (for short, a $\PD_m$ algebra) 
if there exists a $\k$-linear map $\varepsilon\colon A^m \to \k$ 
such that all the bilinear forms $A^i \otimes_{\k} A^{m-i}\to \k$, 
$a\otimes b\mapsto \varepsilon (ab)$ are non-singular. 
In this case, $\varepsilon$ is an isomorphism, 
$A^i=0$ for $i>m$, and $b_i(A)=b_{m-i}(A)$. 

Now suppose $M$ is a compact, connected, oriented $m$-dimensional 
manifold; then, by Poincar\'{e} duality, the cohomology algebra 
$H^{\hdot}(M,\k)$ is a $\PD_m$ algebra over $\k$, with 
the homomorphism  $\varepsilon \colon  
H^{m}(M,\k) = H_m(M,\Z)\otimes \k \to \k$ being determined 
by the orientation class $[M]\in H_m(M,\Z)$ by setting 
$\varepsilon([M] \otimes 1)=1$.

We now restrict to the case $m=3$, so that $M$ is a closed, oriented 
$3$-manifold. Then the cohomology algebra $A=H^{\hdot}(M,\C)$ 
is a Poincar\'e duality $\C$-algebra of dimension~$3$.  Two such $\PD_3$ 
algebras are isomorphic if and only if the corresponding $3$-forms are isomorphic, 
see \cite{Su-poinres}. 

Let $S=\Sym(A_1)$ be the symmetric 
algebra on $A_1=H_1(M,\C)$, which we will identify as before with 
the polynomial ring $\C[x_1,\dots,x_n]$.  In our situation, the chain 
complex from \eqref{eq:univ aomoto} has the form 
\begin{equation}
\label{eq:koszul 3mfd}
\xymatrixcolsep{22pt}
\xymatrix{
A^0\otimes_\C S \ar^{\delta^{0}_A}[r] 
& A^1\otimes_\C S \ar^{\delta^{1}_A}[r] 
&A^2\otimes_\C  S\ar^{\delta^{2}_A}[r] 
&A^3\otimes_\C  S\, ,}
\end{equation}
where the $S$-linear differentials are given by 
$\delta^q_A(u)=\sum_{j=1}^{n} e_j u \otimes x_j$ for $u\in A^q$.  
In our chosen basis, the matrix of $\delta^0$ is $\big(x_1 \, \cdots\, x_n\big)$,  
the matrix of $\delta^2_A$ is the transpose of $\delta^0$, while the matrix 
of $\delta^1_A$ is the $n\times n$ skew-symmetric matrix of linear forms in 
the variables of $S$, with entries given by %
\begin{equation}
\label{eq:delta1 3m}
\delta^1_A(e_i)=
\sum_{j=1}^{n} \sum_{k=1}^{n}\mu_{jik} e_k^{\vee} \otimes x_j\, .
\end{equation}

We wish to describe the resonance varieties $\RR^i_k(M)=\RR^i_k(A)$. 
To avoid trivialities, we will assume for the rest of this section that $n\ge 3$, 
since, otherwise $\RR^i_k(M)\subseteq \{\zero\}$.  Furthermore, we may 
assume that $i=1$; indeed, as shown in \cite[Proposition~6.1]{Su-poinres}, 
\begin{equation}
\RR^{2}_k(M)=\RR^1_k(M)\quad \text{for $1\le k \le n$}, 
\end{equation}
while $\RR^i_0(M)=H^1(M,\C)$, $\RR^{3}_1(M)=\RR^0_1(M)=\{\zero\}$, 
and $\RR^i_k(M)=\myempty$, otherwise.   Now, by \eqref{eq:r1ka}, the 
resonance variety $\RR^1_k(M)$ 
is the vanishing locus of the ideal of codimension $k$ minors 
of the matrix $\delta_M:=\delta^1_A$;  that is, 
\begin{equation}
\label{eq:r1theta}
\RR^1_k(M)=V(I_{n-k}(\delta_M)).
\end{equation}

The rank of a $3$-form $\mu\colon \bigwedge^3 U\to \C$  on a finite-dimensional 
$\C$-vector space $U$  is the 
minimum dimension of a subspace $W\subset U$ such that $\mu$ factors
through $\bigwedge^3 W$.

\begin{proposition}[\cite{Su-poinres}]
\label{prop:rvanish}
If $n\ge 3$ and $\mu_M$ has rank $n=b_1(M)$, 
then 
\[
\RR^1_{n-2}(M)=\RR^1_{n-1}(M)=\RR^1_{n}(M)=\{\zero\}.
\]
\end{proposition}

\subsection{Pfaffians and resonance}
\label{subsec:pfaff}
For a skew-symmetric matrix $\theta$, we shall denote by 
$\Pf_{2r}(\theta)$ the ideal of $2r \times 2r$ Pfaffians of $\theta$.  

\begin{proposition}[\cite{Su-poinres}]
\label{prop:be}
The following hold:
\begin{align} 
\label{eq:res theta}
&\RR^1_{2k}(M) =\RR^1_{2k+1}(M) =V(\Pf_{n-2k}(\delta_M)),&& \text{if $n$ is even},\\ \notag
&\RR^1_{2k-1}(M) =\RR^1_{2k}(M) =V(\Pf_{n-2k+1}(\delta_M)),&& \text{if $n$ is odd}.
\end{align}
\end{proposition}

The skew-symmetric matrix $\delta_M$ is singular, 
since the vector $(x_1,\dots, x_n)$ is in its kernel.  Hence, both its 
determinant $\det(\delta_M)$ and its Pfaffian $\pf(\delta_M)$ vanish. 
In \cite[Ch.~III, Lemmas 1.2 and 1.3.1]{Tu}, Turaev shows how 
to remedy this situation, so as to obtain well-defined determinant 
and Pfaffian polynomials for the $3$-form $\mu_M$. Let 
$\delta_M(i;j)$ be the sub-matrix obtained from $\delta_M$ 
by deleting the $i$-th row and $j$-th column. 

\begin{lemma}[\cite{Tu}]
\label{lem:turaev}
Suppose $n\ge 3$.  There is then a polynomial 
$\Det(\mu)\in S$ such that $\det \delta_M(i;j) = (-1)^{i+j}x_ix_j \Det(\mu)$.  
Moreover, if $n$ is even, then $\Det(\mu)=0$, while if $n$ is odd, 
then $\Det(\mu)=\Pf(\mu)^2$, where $\pf (\delta_M(i;i)) = (-1)^{i+1} x_i \Pf(\mu)$. 
\end{lemma}

\subsection{The top resonance variety}
\label{subsec:res3-top}

We will need in the sequel the notion of `generic' alternating $3$-form, 
introduced and studied by Berceanu and Papadima in \cite{BP}.  For 
our purposes, it will be enough to consider the case when $n=2g+1$,  
for some $g\ge 1$.  We say that a $3$-form $\mu_A$ is 
{\em generic}\/ if there is an element $c\in A^1$ such that the $2$-form 
$\gamma_c\in A_1\wedge A_1$ defined by 
\begin{equation}
\label{eq:2form}
\gamma_c(a \wedge b)=\mu_A(a\wedge b\wedge c) \quad\text{for $a,b\in A^1$}
\end{equation} 
 has rank $2g$, that is, $\gamma^{g}_c\ne 0$ 
in $\bigwedge^{2g} A_1$.  Equivalently, in a suitable basis for $A^1$, 
we may write
$\mu_A=\sum_{i=1}^g a_i  b_i  c+ 
\sum q_{ijk}\, z_i z_j  z_k$, 
where each $z_i$  belongs to the span of $a_1,b_1,\dots , a_g, b_g$ in $A^1$, 
and the coefficients $q_{ijk}$ are in $\C$.  

\begin{example}
\label{ex:surf}
Let $M=\Sigma_g \times S^1$ be the product of a circle with a 
closed, orientable surface of genus $g\ge 1$.   If $a_1,b_1,\dots , a_g, b_g$ is  
the standard symplectic basis for $H_1(\Sigma_g,\Z)=\Z^{2g}$, and  
$c$ generates $H^1(S^1,\Z)=\Z$, then $\mu_M=\sum_{i=1}^g a_i b_i c$, 
and so $\mu_M$ is generic. A routine 
computation shows that $\Pf(\mu_M)=x_{2g+1}^{g-1}$.  
Furthermore, $\RR^1_k(M)=\{x_{2g+1}=0\}$ for $1\le k\le 2g-2$ and 
$\RR^1_{2g-1}(M)=\{\zero\}$. 
\hfill $\Diamond$
\end{example}

More generally, if $M$ is homology cobordant to $S^1\times \varSigma_g$ 
for some $g\ge 1$, then $\mu_M=\mu_{S^1\times \Sigma_g}$ is generic.  
For $n=3$ and $n=5$, there 
is a single irreducible $3$-form of rank $n$ (up to isomorphism), and that 
form is of the type just discussed.  In rank $n=7$, there are $5$ irreducible 
$3$-forms: two generic, $\mu=e_7(e_1e_4+e_2e_5+e_3e_6)$ and 
$\mu'=\mu+e_4e_5e_6$, and three non-generic. 
  
Using \cite[Theorem~8.6]{Su-poinres}, we obtain the following 
description of the top resonance variety of a closed, orientable $3$-manifold.

\begin{theorem}
\label{thm:res closed3m}
Let $M$ be a closed, orientable $3$-manifold.  Set $n=b_1(M)$ 
and let $\mu_M$ be the associated $3$-form.  Then 
\begin{equation}
\label{eq:r1 3m}
\RR^1_1(M)= 
\begin{cases}
\emptyset & \text{if\/ $n=0$};\\
\{0\} & \text{if\/ $n=1$ or $n=3$ and $\mu_M$ has rank $3$};\\
V(\Pf(\mu_M))  & \text{if\/  $n$ is odd, $n>3$, and $\mu_M$ is generic};\\
H^1(M,\C) & \text{otherwise}.
\end{cases}
\end{equation}
\end{theorem}

\begin{remark}
\label{rem:be}
The case when $b_1(M)$ is even and positive is worth dwelling upon. 
In this case, the equality $\RR^1(M)=H^1(M,\C)$ was first proved in \cite{DS},  
where it was used to show that 
the only $3$-manifold groups which are also K\"{a}hler 
groups are the finite subgroups of ${\rm O}(4)$.  Another application 
of this equality was given in \cite{PS-forum}: if $M$ is a closed, orientable 
$3$-manifold such that $b_1(M)$ is even and $M$ fibers over the circle, 
then $M$ is not $1$-formal.
\hfill $\Diamond$
\end{remark}

\section{Alexander polynomials and characteristic varieties of $3$-manifolds}
\label{sect:alex-3dim}

In this section, we collect some facts regarding the Alexander 
polynomials and the characteristic varieties of closed, orientable, 
$3$-dimensional manifolds.

\subsection{Poincar\'e duality and characteristic varieties}
\label{subsec:pd}

Let $M$ be a smooth, closed, orientable manifold of dimension $m$. 
By Morse theory, $M$ admits a finite cell decomposition;  consequently, 
its fundamental group, $\pi=\pi_1(M)$ admits a finite presentation. 
The involution $g\mapsto g^{-1}$ taking each element of $\pi$ 
to its inverse induces an algebraic automorphism of $\Hom (\pi, \C^*)$, 
taking a character $\rho$ to the character $\bar\rho$ given by 
$\bar\rho(g)=\rho(g^{-1})$. 

\begin{proposition}
\label{prop:cv man}
The above automorphism of $\Hom (\pi, \C^*)$ restricts to isomorphisms  
\[
\VV^{i}_k(M)\cong \VV^{m-i}_k(M),
\] 
for all $i\ge 0$ and $k\ge 0$.
\end{proposition}
  
\begin{proof}
Poincar\'e duality with local coefficients (see e.g.~\cite[\S 2]{W}) yields  
isomorphisms $H^{i}(M,\C_{\rho})\cong  H_{m-i}(M,\C_{\bar\rho})$.  The 
claim follows.
\end{proof}

A well-known theorem of E.~Moise insures that every $3$-manifold 
has a smooth structure. Thus, the above proposition together with the 
discussion from \S\ref{subsec:cv} yield the following corollary.

\begin{corollary}
\label{cor:cv3d}
Let $M$ be a closed, orientable $3$-manifold. 
Then  $\VV^i_0(M)=H^1(M,\C^*)$ for all $i\ge 0$, 
$\VV^{3}_1(M)=\VV^0_1(M)=\{\one\}$, 
\begin{equation*}
\VV^{2}_k(M)=\VV^1_k(M)\quad \text{for $1\le k \le b_1(M)$},
\end{equation*}
and otherwise $\VV^i_k(M)=\myempty$.  
\end{corollary}

Thus, in order to compute the characteristic varieties of a 
$3$-manifold $M$ as above, it is enough 
to determine the sets $\VV^1_k(M)$ for $1\le k\le b_1(M)$.  

\subsection{The Alexander polynomial of a closed $3$-manifold}
\label{subsec:alex 3d}

As before, let $H$ be the quotient of $H_1(M,\Z)$ by its 
torsion subgroup.  We will identify the group ring 
$\Z[H]$ with the ring of Laurent polynomials
$\Z[t_1^{\pm 1},\dots , t_n^{\pm 1}]$, where $n=b_1(M)$. 
For a Laurent polynomial $\lambda$, 
we will denote by $\bar{\lambda}$ its image 
under the involution $t_i\mapsto t_i^{-1}$, and 
say that $\lambda$ is {\em symmetric}\/ if   
$\lambda \doteq \bar{\lambda}$.

Assume now that $M$ is $3$-dimensional, and let 
$\Delta_M\in \Z[H]$ be its Alexander polynomial. Recall that 
$\Delta_M$ is only defined up to units, i.e., up to multiplication 
by monomials $\pm t_1^{a_1}\cdots t_n^{a_n}$ with $a_i\in \Z$. 
Work of Milnor \cite{Mi62} and Turaev \cite{Tu75, Tu} shows that 
$\Delta_M$ is symmetric. 

Conversely, if $H=\Z$ or $\Z^2$, then every symmetric Laurent polynomial 
$\lambda\in \Z[H]$ can be realized as the Alexander polynomial of a closed, 
orientable $3$-manifold $M$ with $H_1(M,\Z)=H$; see \cite[VII.5.3]{Tu}.
Furthermore, every symmetric Laurent polynomial $\lambda$ in $n\le 3$ 
variables such that $\lambda(\one)\ne 0$ can be realized as the Alexander 
polynomial of a closed, orientable $3$-manifold $M$ with $b_1(M)=n$; 
see \cite{Az}.

On the other hand, for $n\ge 4$, the situation is quite different.

\begin{theorem}[\cite{Tu}] 
\label{thm:az-bis}
If $M$ is a closed, orientable  $3$-manifold $M$ with $b_1(M)\ge 4$,  
then $\Delta_M(\one)=0$.
\end{theorem}

The theorem follows at once from \cite[\S{II}, Corollaries 2.2 and 5.2.1]{Tu}.  
As an application, we deduce that $\Delta_M\ne 1$ if $b_1(M)\ge 4$, 
a result which is also proved in \cite[Theorem 8]{Az} by different means.

\subsection{Characteristic varieties and the Alexander polynomial}
\label{subsec:alex cv}
Now consider the maximal torsion-free abelian cover $M^H\to M$, 
and let $A_M= H_1(M^H,\Z)$, viewed as a $\Z[H]$-module 
as in  \S \ref{subsec:alexpoly}. Recall that the determinantal ideal 
$E_1(A_M)$ is always contained in the ideal $\Delta_{M}\cdot I_H$, 
where $I_H=\ker(\varepsilon\colon \Z[H]\to \Z)$ is the augmentation 
ideal, provided $b_1(M)\ge 2$. In \cite[Theorem 5.1]{McM}, 
McMullen established a closer relationship between these 
ideals, in the case when $M$ is a closed, orientable $3$-manifold $M$  
(see also Turaev \cite[Theorem II.1.2]{Tu}).

\begin{theorem}[\cite{McM}]
\label{thm:mcmullen}
Let $n=b_1(M)$. Then
\begin{equation}
\label{eq:mcm}
E_1(A_M) = \begin{cases} 
(\Delta_M) & \text{if $n\le 1$},\\[2pt]
I^2_H\cdot (\Delta_M) & \text{if $n\ge 2$}.
\end{cases}
\end{equation}
\end{theorem}

Recall now that $\ZZ^1_1(M)$ denotes the intersection of the characteristic 
variety $\VV^1_1(M)$ with the identity component of the character group, 
$\Char^0(M)=(\C^*)^n$.  

\begin{proposition}
\label{prop:cv 3d}
Let $M$ be a closed, orientable, $3$-dimensional manifold.  Then 
\begin{equation}
\label{eq:vdel}
\ZZ^1_1(M)= V(\Delta_M) \cup \{\one\}.
\end{equation}
Moreover,  if $b_1(M)\ge 4$,  then $\ZZ^1_1(M)= V(\Delta_M)$.  
\end{proposition}

\begin{proof}
The first equality follows at once from Proposition \ref{prop:zz1-pi}
and Theorem \ref{thm:mcmullen}.   If $b_1(M)\ge 4$, the second 
equality follows from the first one and Theorem \ref{thm:az-bis}.
\end{proof}

\begin{remark}
\label{rem:zzvv}
If the group $H_1(M,\Z)$ has non-trivial torsion, 
the inclusion $\ZZ^1_1(M)\subseteq \VV^1_1(M)$ may very  
well be strict.  A rich source of examples illustrating this 
phenomenon is provided by Seifert fibered manifolds (see  
Example \ref{ex:sigma248} below).
\end{remark}

\begin{corollary}
\label{cor:tcv}
Let $M$ be a closed, orientable, $3$-dimensional manifold, 
and set $W=\VV^1_1(M)$.  
\begin{enumerate}
\item \label{dv1}
If $\Delta_M(\one) \ne 0$, then 
$\tau_{\one}(W)=\TC_{\one}(W)=\{\one\}$.
\item \label{dv2}
If $\Delta_M(\one) =0$, yet $\Delta_M\ne 0$, then 
$\tau_{\one}(W)= \tau_{\one}(V(\Delta_M)) $ and 
$\TC_{\one}(W)= \TC_{\one}(V(\Delta_M))$.
\item \label{dv3}
If $\Delta_M= 0$, then 
$\tau_{\one}(W)=\TC_{\one}(W)=H^1(M,\C)$.
\end{enumerate}
Moreover,  if $b_1(M)\ge 4$,  then case \eqref{dv1} does not occur.
\end{corollary}

\begin{proof}
Recall from \ref{subsec:exp tc} that both $\tau_{\one}(W)$ and $\TC_{\one}(W)$ 
depend only on the analytic germ of $W$ around the identity $\one \in \Char(M)^0$. 
Thus, in computing these tangent cones at $\one$, we may replace $W$ by 
$W\cap  \Char(M)^0=\ZZ^1_1(M)$. All the claims now follow directly from 
Proposition \ref{prop:cv 3d}.
\end{proof}

\section{A Tangent Cone theorem for $3$-manifolds}
\label{sect:3dim}

For closed, orientable, $3$-dimensional manifolds, 
the Tangent Cone theorem takes a rather surprisingly 
concrete form, which we proceed to describe in this section.

\subsection{A $3$-dimensional Tangent Cone theorem}
\label{subsec:tcone 3-mfd}

We start by isolating a class of closed $3$-manifolds 
for which the full Tangent Cone formula \eqref{eq:tc} 
holds in degree $i=1$ and depth $k=1$. 

\begin{lemma}
\label{lem:md0}
Let $M$ be a closed, orientable, $3$-dimensional manifold 
such that  $\Delta_M=0$.  Then 
$\tau_{\one}(\VV^1_1(M))=\TC_{\one}(\VV^1_1(M))= \RR^1_1(M)=H^1(M,\C)$. 
\end{lemma}

\begin{proof}
By case \eqref{dv3} of Corollary \ref{cor:tcv}, we have that 
$\tau_{\one}(\VV^1_1(M))=\TC_{\one}(\VV^1_1(M))= H^1(M,\C)$. 
On the other hand, by Corollary \ref{cor:tcone inc}, we always have 
$\TC_{\one}(\VV^1_1(M)) \subseteq \RR^1_1(M)$, while 
$\RR^1_1(M)\subseteq H^1(M,\C)$ by definition. The 
claim follows.
\end{proof}

\begin{example}
\label{ex:s1s2}
Let $M=\connsum_1^n S^1\times S^2$.  Then clearly 
$\mu_M=0$ and $\Delta_M=0$; in particular, Lemma \ref{lem:md0}  applies.
In fact, $M$ is formal, and so the Tangent Cone formula holds 
 in all degrees and depths.
\end{example}

The next result shows that the second half of the Tangent Cone formula 
holds for a large class of closed $3$-manifolds with odd first 
Betti number (regardless of whether these manifolds are $1$-formal 
or not), yet fails for most $3$-manifolds with even first Betti number.

\begin{theorem}
\label{thm:tc 3d}
Let $M$ be a closed, orientable $3$-manifold, and 
set $n=b_1(M)$. 
\begin{enumerate}
\item \label{odd} 
If $n\le 1$, or $n$ is odd, $n\ge 3$, and $\mu_M$ 
is generic, then $\TC_{\one}(\VV^1_1(M))=\RR^1_1(M)$.
\item \label{even} 
If $n$ is even, $n\ge 2$, then 
$\TC_{\one}(\VV^1_1(M))= \RR^1_1(M)$ if 
and only if $\Delta_M= 0$.
\end{enumerate}
\end{theorem}

\begin{proof}
\eqref{odd} 
If $n\le 1$, the Tangent Cone formula always holds. If 
$n= 3$, then our genericity assumption implies that $\mu_M=e_1e_2e_3$ 
in a suitable basis for $H^1(M,\C)$.  It follows that $\RR^1_1(M)=\{\zero\}$, 
and so $\TC_{\one}(\VV^1_1(M))=\{\zero\}$, too, by Theorem \ref{thm:lib}.

So let assume that $n$ is odd and $n>3$. 
In this case, Proposition \ref{prop:cv 3d} insures that 
$\ZZ^1_1(M)=V(\Delta_M)$.  Moreover, as noted previously, $\VV^1_1(M)$ 
coincides with $\ZZ^1_1(M)$ around the identity, and so the two 
varieties share the same tangent cone at $\one$. 

Now, as explained in \S\ref{subsec:exp tc}, 
$\TC_{\one}(V(\Delta_M))$ is the variety defined by 
the homogeneous polynomial $\init(\widetilde\Delta_M)$, where 
$\widetilde\Delta_M(x_1,\dots,x_n)=\Delta_M(x_1+1,\dots,x_n+1)$. 
Putting things together, we conclude that
\begin{equation}
\label{eq:tcone vdel}
\TC_{\one}(\VV^1_1(M))= V(\init (\widetilde\Delta_M)).
\end{equation}

On the other hand, as shown by Turaev in 
\cite[Theorem III.2.2]{Tu}, for $n\ge 3$ and $n$ odd, 
we have that
\begin{equation}
\label{eq:inwide}
\init(\widetilde\Delta_M)=\Det(\mu_M).
\end{equation}

We also know from Lemma \ref{lem:turaev} that $\Det(\mu)=\Pf(\mu)^2$; 
hence, $V(\Det(\mu_M)=V(\Pf(\mu_M))$.  
Finally, since $n$ is odd, $n>3$, and $\mu_M$ is generic, 
Theorem \ref{thm:res closed3m} implies that 
$V(\Pf(\mu_M))=\RR^1_1(M)$. Combining the aforementioned 
equalities, we conclude that 
\[
\TC_{\one}(\VV^1_1(M)) =V(\init (\widetilde\Delta_M))=V(\Det(\mu_M)=V(\Pf(\mu_M))=
\RR^1_1(M).
\]
\eqref{even} 
Now suppose that $n$ is even and $n\ge 2$.  By 
Theorem \ref{thm:res closed3m}, we have that $\RR^1_1(M)=\C^n$.  
On the other hand, by Corollary \ref{cor:tcv}, the following alternative holds:
if $\Delta_M=0$, then  $\TC_{\one}(\VV^1_1(M))$ also equals $\C^n$;  
otherwise  $\TC_{\one}(\VV^1_1(M))$ is a proper subvariety of 
$\C^n$.  This completes the proof.
\end{proof}

\subsection{Algebraic models for $3$-manifolds}
\label{subsec:models3d}

As an application of the techniques developed so far, we derive 
a partial characterization of the formality and finiteness 
properties for rational models of $3$-manifolds.

\begin{theorem}
\label{thm:even betti1}
Let $M$ be a closed, orientable, $3$-dimensional manifold, and set 
$n=b_1(M)$. 

\begin{enumerate}
\item \label{f1}
If $n\le 1$, then $M$ is formal, and has the rational homotopy type 
of $S^3$ or $S^1\times S^2$. 
\item \label{f2} 
If $n$ is even, $n\ge 2$, and $\Delta_M\ne 0$, then 
$M$ is not $1$-formal. 
\item \label{f3} 
If $\Delta_M\ne 0$, yet $\Delta_M(\one) =0$  and $\TC_{\one}(V(\Delta_M))$ 
is not a finite union of rationally defined linear subspaces, then $M$ 
admits no $1$-finite $1$-model. 
\end{enumerate}
\end{theorem}

\begin{proof}
\eqref{f1}
As mentioned previously, any connected CW-complex $X$ with finite $2$-skeleton 
and with $b_1(X)\le 1$ is $1$-formal. On the other hand, if $M$ is a closed, orientable 
$3$-manifold, then $1$-formality is equivalent to formality, see \cite{FM05}. Thus, if 
$b_1(M)=0$ or $1$, then $M$ is formal, and so, as noted in \cite{PS-formal}, $M$ 
must be rationally homotopy equivalent to either $S^3$ or $S^1\times S^2$. 

\eqref{f2} 
Now suppose $b_1(M)$ is even and positive, and $\Delta_M\ne 0$. 
Then, by part \eqref{even} of Theorem \ref{thm:tc 3d}, we have that 
$\TC_{\one}(\VV^1_1(M))\ne  \RR^1_1(M)$. Thus, by 
Corollary \ref{cor:tcone}, $M$ is not $1$-formal. 

\eqref{f3} 
Finally, if $\Delta_M\ne 0$ and $\Delta_M(\one) =0$, then, by Corollary \ref{cor:tcv}, 
$\tau_{\one}(\VV^1_1(M))= \tau_{\one}(V(\Delta_M)) $ and 
$\TC_{\one}(\VV^1_1(M))= \TC_{\one}(V(\Delta_M))$.
On the other hand, if not all the irreducible components of 
$\TC_{\one}(V(\Delta_M))$ are linear subspaces defined over $\Q$, 
then, by Lemma \ref{lem:exp-tcone}, 
$\tau_{\one}(V(\Delta_M))\ne \TC_{\one}(V(\Delta_M))$. 
Therefore, if both assumptions are satisfied, 
$\tau_{\one}(\VV^1_1(M))$ is a proper subset 
of $\TC_{\one}(\VV^1_1(M))$, 
and so, by Theorem \ref{thm:tcone-fm}, 
$M$ cannot have a $1$-finite $1$-model. 
\end{proof}

Now let $\pi=\pi_1(M)$ be the fundamental group of $M$, 
and let $\m=\m(\pi)$ be its Malcev Lie algebra.  In the 
three cases treated in Theorem \ref{thm:even betti1}, 
the following hold:
\begin{enumerate}
\item[\eqref{f1}]  \label{m1}
$\m=0$ (if $n=0$) or $\m=\Q$ (if $n=1$).
\item[\eqref{f2}]   \label{m2}
$\m$ is not the LCS completion of a finitely generated, quadratic Lie algebra.
\item[\eqref{f3}]   \label{m3}
$\m$ is not the LCS completion of a finitely presented Lie algebra.  
\end{enumerate}
\subsection{Discussion and examples}
\label{subsec:discuss}

If $b_1(M)=2$, then all three possibilities laid out 
in Corollary \ref{cor:tcv} do occur.

\begin{example}
\label{fm}
Let $M=S^1\times S^2 \# S^1\times S^2$;  
then $\Delta_M= 0$, and so 
$\TC_{\one}(\VV^1_1(M))=\RR^1_1(M)=\C^2$.
Clearly, the manifold $M$ is formal. 
\hfill $\Diamond$
\end{example}

\begin{example}
\label{ex:notfm-finite}
Let $M$ be the Heisenberg $3$-dimensional nilmanifold;  
then $\Delta_M=1$ and $\mu_M=0$, and so 
$\TC_{\one}(\VV^1_1(M))=\{\zero\}$,  whereas $\RR^1_1(M)=\C^2$. 
The manifold $M$ admits a finite model, namely, $A=\bigwedge(a,b,c)$ 
with $\D a=\D b=0$ and $\D c=ab$, but $M$ is not $1$-formal. 
\hfill $\Diamond$
\end{example}

\begin{example}
\label{ex:finite}
Consider the symmetric Laurent polynomial 
$\lambda=(t_1+t_2)(t_1t_2+1)-4t_1t_2$.  By the discussion from 
\S\ref{subsec:alex 3d}, there is a closed, orientable $3$-manifold 
$M$ with $H_1(M,\Z)=\Z^2$ and $\Delta_M=\lambda$.  It is readily 
seen that $\tau_{\one}(\VV^1_1(M))=\{\zero\}$, which is a proper 
subset of $\TC_{\one}(\VV^1_1(M))=\{x_1^2+x_2^2=0\}$. Note 
that the latter variety decomposes as the union of two lines defined 
over $\C$, but not over $\Q$; hence, $M$ admits no $1$-finite $1$-model. 
\hfill $\Diamond$
\end{example}

Now consider the case when $n=b_1(M)$ is odd and at least $3$, and 
$\mu_M$ is not generic, a case which is not covered by Theorem \ref{thm:tc 3d}.  
In this situation, $\RR^1_1(M)=H^1(M,\C)$, by Theorem \ref{thm:res closed3m}, 
while the equality $\TC_{\one}(\VV^1_1(M))=\RR^1_1(M)$ 
may or may not hold.  For instance, if $M$ is the connected sum of 
$n$ copies of $S^1\times S^2$, then $\mu_M=0$ is not generic, 
yet the aforementioned equality holds (see Corollary \ref{cor:cjl conn} 
below for a more general instance of this phenomenon).  On the other 
hand, as we shall see in Example \ref{ex:bdry gen}, there are $3$-manifolds 
$M$ with $n=15,  21, 45, 55, 91, \ldots$ for which $\mu_M$ is not generic,  
while $\TC_{\one}(\VV^1_1(M))$ is a proper subset of $\RR^1_1(M)$. 

\section{Connected sums}
\label{sect:resconn}

Let $M=M_1\,\#\, M_2$ be the connected sum of two closed, orientable 
manifolds of dimension $m\ge 3$.  By the van Kampen theorem, the fundamental 
group of $M$ splits as a free product, $\pi_1(M)=\pi_1(M_1)*\pi_1(M_2)$, 
from which we get a direct product decomposition of the corresponding 
character tori,
\begin{equation}
\label{eq:char-conn}
\Char(\pi_1(M))=\Char(\pi_1(M_1))\times \Char(\pi_1(M_2)).
\end{equation}

Likewise, we have that $H^1(M,\C)=H^1(M_1,\C)\times H^1(M_2,\C)$.
The next result describes the behavior of the cohomology 
jump loci under these decompositions. 

\begin{theorem}
\label{thm:jump conn}
Let $M=M_1\,\#\, M_2$ be the connected sum of two closed, orientable, smooth 
$m$-manifolds, $m\ge 3$.  Then,  for $i=1$ or $m-1$ and for all $k\ge 0$, 
\begin{align*}
\label{eq:v1r1conn}
 \VV^i_k(M)&=
\bigcup\limits_{r+s=k-1} \VV_r^i(M_1) \times \VV_s^i(M_2),
 &\RR^i_k(M)&=
\bigcup\limits_{r+s=k-1} \RR_r^i(M_1) \times \RR_s^i(M_2),\\
\intertext{while, for $1<i<m$,}
\VV^i_k(M)&=
\bigcup\limits_{r+s=k} \VV_r^i(M_1) \times \VV_s^i(M_2),
 &\RR^i_k(M)&=
\bigcup\limits_{r+s=k} \RR_r^i(M_1) \times \RR_s^i(M_2).
\end{align*}
\end{theorem}

\begin{proof}
The claims involving resonance varieties are proved in 
\cite[Proposition~5.4]{Su-poinres}. 
A completely similar proof works for the characteristic varieties. 
\end{proof}

Staying with the same notation, we obtain the following corollary 
regarding the compatibility of the Tangent Cone formula (at least of 
its second half)  with respect to connected sums. 

\begin{corollary}
\label{lcor:tcone conn}
Suppose that $\TC_{\one}(\VV^i_s(M_j)) = \RR^i_s(M_j)$ 
for $j=1,2$, in some fixed degree $0<i<m$, and in depths $s<k$ 
if $i=1$ or $m-1$, or $s\le k$ otherwise.   Then 
$\TC_{\one}(\VV^i_k(M_1\# M_2)) = \RR^i_k(M_1\# M_2)$.
\end{corollary}

In degree $i=1$ and depth $k=1$, Theorem \ref{thm:jump conn} 
yields  another corollary, the  
conclusions of which can also be deduced from \cite[Lemma 9.8]{DPS-duke} 
and \cite[Lemma 5.2]{PS-mathann}, respectively. 

\begin{corollary}
\label{cor:cjl conn}
Let $M=M_1\,\#\, M_2$ be the connected sum of two closed, orientable,  
smooth $m$-manifolds ($m\ge 3$) with $b_1(M_1)$ and $b_1(M_2)$ 
both non-zero.
Then $\VV^1_1(M)=H^1(M,\C^*)$ and $\RR^1_1(M)=H^1(M,\C)$. 
\end{corollary}

In particular, the full Tangent Cone formula in this degree and depth, 
$\tau_{\one}(\VV^1_1(M))=\TC_{\one}(\VV^1_1(M)) = \RR^1_1(M)$, 
holds for manifolds which admit a connected sum decomposition  as above. 
Combining this corollary with Lemma \ref{lem:delta cv}, part \eqref{dc1}, 
we obtain the following---presum\-ably well-known---application. 

\begin{corollary}
\label{lem:cjl conn}
Let $M=M_1\,\#\, M_2$ be the connected sum of two closed, 
smooth $m$-manifolds ($m\ge 3$) with non-zero first Betti number.  
Then $\Delta_M=0$. 
\end{corollary}

A classical theorem of J.~Milnor insures that every closed, orientable 
$3$-manifold decomposes as the connected sum of finitely many 
irreducible $3$-manifolds. Since every $3$-manifold is smooth, 
Theorem \ref{thm:jump conn}
reduces the computation of the cohomology jump loci of arbitrary 
closed, orientable $3$-manifolds to that of irreducible ones.

\section{Graph manifolds}  
\label{sect:graphman}

In this section we study in more detail the cohomology jump loci 
and the formality properties of certain classes of graph manifolds.  
We start with a look at the Seifert fibered spaces, which are the basic 
building blocks for such manifolds. 

\subsection{Seifert manifolds}
\label{subsec:seifert}

A compact $3$-manifold is a Seifert fibered space if and only if it is 
foliated by circles. One can think of such a manifold $M$ as a  
bundle in the category of orbifolds, in which the circles of the 
foliation are the fibers, and the base space of the orbifold bundle is the 
quotient space of $M$ obtained by identifying each circle to a point. 
We refer to \cite{Sc} as a general reference for the subject. 

For our purposes here, we will only consider closed, orientable Seifert manifolds 
with orientable base.   Every such manifold $M$ admits an effective circle action, 
with orbit space a Riemann surface $\Sigma_g$, and finitely many 
exceptional orbits, encoded in pairs of coprime integers $(\alpha_1,\beta_1), 
\dots,  (\alpha_s,\beta_s)$ with $\alpha_j\ge 2$. The fundamental group  
$\pi=\pi_1(M)$ admits a presentation of the form 
\begin{equation}
\label{eq:seifert-pi1}
 \begin{split}
\pi&=\big\langle x_1, y_1,\dots, x_g, y_g, z_1,\dots,z_s, h\mid  
\text{$h$ central}, \\
&\qquad [x_1,y_1]\cdots[x_g,y_g]z_1\cdots z_s=h^{b},\:\: 
 z_1^{\alpha_1}h^{\beta_1}=\cdots =  z_s^{\alpha_s}h^{\beta_s} =1\big\rangle,
 \end{split}
 \end{equation}
where the integer $b$ encodes the obstruction to trivializing 
the bundle $p\colon M\to \Sigma_g$ outside tubular neighborhoods 
of the exceptional orbits.  

Let $e=-\big(b+\sum_{i=1}^{s}\beta_i/\alpha_i\big)$ be 
the Euler number of the orbifold bundle. 
If $g=0$, then $b_1(M)= 0$ or $1$, according to whether $e\ne 0$  
or $0$; therefore, by Theorem \ref{thm:even betti1}\eqref{f1}, 
the manifold $M$ is formal.  So let us assume 
that $g>0$.  Then $M$ admits a finite-dimensional  model, 
$A = \big( H^{\hdot} (\Sigma; \Q) \otimes_{\Q} \bwedge (c), \D \big)$, 
where $\deg c=1$ and the differential $\D$ is defined as follows: 
$\D=0$ on $H^{\hdot} (\Sigma; \Q)$, while $\D c=0$ if $e=0$  
and $\D c=\omega$, where $\omega\in H^{2} (\Sigma; \Q)$ 
is the orientation class, otherwise.  As shown in \cite{PS-imrn, SW},  
the Malcev Lie algebra $\m(\pi)$ is the LCS completion of graded 
algebra with relations in degrees $2$ and $3$; furthermore, $\pi$ is 
$1$-formal if and only if $e=0$. 

The simplest Seifert manifold with $e=0$ is the product 
$M=\Sigma_g \times S^1$ ($g\ge 1$) from Example \ref{ex:surf};  
in this case, $\VV^1_k(M)=\{t\in (\C^*)^{2g+1} \mid t_{2g+1}=1\}$ 
for $1\le k\le 2g-2$ and $\VV^1_{2g-1}(M)=\{\one\}$. 
On the other hand, if $e\ne 0$, then, as shown in \cite{PS-imrn}, 
the morphism  $p^*\colon H^1(\Sigma_g,\C^*)\to H^1(M,\C^*)$ 
induced by the orbit map $p\colon M\to \Sigma_g$ 
defines an isomorphism of analytic germs, 
$\VV^1_k (\Sigma_g)_{(\one)}\cong \VV^1_k (M)_{(\one)}$, 
for each $k\ge 0$. 

On the other hand, if $H_1(M,\Z)$ has torsion, then the corresponding 
connected components of $H^1(M,\C^*)$ may contain irreducible 
components of $\VV^1_1(M)$ which do not pass through $\one$.   
Here is a concrete such example, extracted from  \cite{DPS-imrn, SYZ}.

\begin{example}
\label{ex:sigma248}
Consider the Brieskorn manifold $M=\Sigma(2,4,8)$.
Then $H_1(M,\Z)=\Z^2 \oplus \Z_4$, and so 
$\Char(M)= (\C^*)^2 \times \{\pm 1, \pm i\}$.  
Direct computation shows that $\Delta_M=1$, and so 
$\ZZ^1_1(M)=\{\one\}$, whereas 
$\VV^1_1(M) = \{\one\} \cup  (\C^*)^2 \times \{-1\}$.
\hfill $\Diamond$
\end{example}

\subsection{Tree graph-manifolds}
\label{subsec:tree-gm}

Every irreducible closed, orientable $3$-manifold $M$ 
admits a JSJ decomposition along  incompressible tori. 
That is to say, there is a finite collection of subtori $T$ with product  
neighborhood $N(T)$, such that each connected component of 
$M\setminus N(T)$ is irreducible. 
A closed, orientable $3$-manifold $M$ is a {\em graph-manifold}\/ if its 
JSJ decomposition consists only of Seifert fibered pieces. 
Associated to such a manifold there is a graph $\Gamma=(V,E)$ 
with a vertex $v$ for each component $M_v$ of $M\setminus N(T)$, and 
with an edge $e=\{v,w\}$ whenever $M_v$ and $M_w$ are glued 
along a torus $T_e$ from $T$. 

In \cite{DH}, Doig and Horn provide an algorithm for 
computing the rational cohomology ring of a closed, orientable 
graph manifold $M$.  For instance, if $M$ is a {\em tree 
graph-manifold}\/ (that is, the underlying graph $\Gamma$ is a tree), 
and all closed-up base surfaces $\Sigma_v$ are orientable, then 
\begin{equation}
\label{eq:coho-tree}
H^{\hdot}(M;\Q)\cong \connsum_{v\in V} H^{\hdot}(\Sigma_v \times S^1;\Q). 
\end{equation}
Thus, if we let $g_v$ be the genus of $\Sigma_v$, the intersection 
form of $M$ can be written, in a suitable basis for $H^1(M,\Z)$, as 
\begin{equation}
\label{eq:mu-tree}
\mu_M=\sum_{v\in V} \sum_{i=1}^{g_v} a_{v,i} b_{v,i} c_v \,.
\end{equation} 

\begin{proposition}
\label{prop:res tree-man}
Let $M$ be a tree graph-manifold with orientable base surfaces.  
Then the resonance varieties $\RR^i_k(M)$ are either empty, 
or equal to $H^1(M,\C)$, or are finite 
unions of coordinate subspaces in $H^1(M,\C)$. 
\end{proposition}

\begin{proof}

By the computation from Example \ref{ex:surf}, we know that 
the claim is true when $M=\Sigma_g \times S^1$.  
The general case follows at once from \eqref{eq:coho-tree} 
and Theorem \ref{thm:jump conn}.
\end{proof}

\begin{example}
\label{ex:DH}
The main result of \cite{DH} is Theorem 6.1, which states 
that not every closed $3$-manifold is homology cobordant 
to a tree graph manifold. The proof reduces to showing that the
intersection forms $\mu=e_1e_2e_3+e_1e_5e_6+e_2e_4e_5$ and  
$\mu'=e_1e_2e_3+e_4e_5e_6$ are not equivalent, up to a change 
of basis in $\GL(6,\Q)$.   This is done in \cite[Theorem~6.4]{DH} 
by a rather long  argument; here is a much shorter proof of this fact.  

A computation recorded in \cite{Su-poinres} shows that 
$\RR^1_2(\mu)  = \{x_1=x_2=x_5=0\}$; on the other hand, 
by Theorem \ref{thm:jump conn},
$\RR^1_2(\mu')  = \{x_1=x_2=x_3=0\} \cup  \{x_4=x_5=x_6=0\}$. 
Thus, the respective resonance varieties are not isomorphic, 
and hence the two intersection forms are not equivalent (over $\Q$). 
\hfill$\Diamond$
\end{example}

\subsection{Boundary manifolds of line arrangements}
\label{subsec:bdry}

Let $\A=\{\ell_0,\dots , \ell_n\}$ be an arrangement of 
projective lines in $\CP^2$.  We associate to $\A$ a graph 
$\Gamma=(V,E)$, with vertex set $V=\A\cup \PP$, where $\PP$ 
are the points $P_J=\bigcap_{j\in J} \ell_j$ where three or more 
lines intersect. The graph $\Gamma$ has an edge from $\ell_i$ to $\ell_j$ if 
those lines  are transverse, and an edge from a multiple point $P$ to 
each line $\ell_i$ on which it lies.

Now let $M=M(\A)$ be the boundary of a regular 
neighborhood of $\A$. Then $M$ is a closed, orientable graph 
manifold, with underlying graph $\Gamma$; 
the vertex manifolds $M_v$ are of the 
form $S^1\times S_v$, where $S_v$ is the $2$-sphere with 
$\deg(v)$ open disks removed, and all the gluing maps 
are flips, i.e., diffeomorphisms of the boundary tori  
given by the matrix $J=\left(\begin{smallmatrix}0&1\\1&0\end{smallmatrix}\right)$.

For instance, if  $\A$ is a pencil of lines defined by 
$\{z_1^{n+1}-z_2^{n+1}=0\}$, then $M=\connsum^{n} S^{1}\times S^{2}$, 
whereas if $\A$ is a near-pencil  defined by
$\{z_0(z_1^n-z_2^n)=0\}$, then $M=S^1\times\Sigma_{n-1}$. 

The group $H_1(M,\Z)$ is free abelian, of rank equal to $n+b_1(\Gamma)$.  
We fix a basis for $H^1(M,\Z)$, consisting of classes  $e_i$ 
dual to the meridians of the lines  $\ell_1,\dots , \ell_n$, as well as classes 
$f_{i,j}$ dual to the cycles in the graph. The latter classes are indexed 
by the set $B$  of pairs $(i,j)$ with $i<j$ for which 
either $\ell_i \pitchfork \ell_j$, 
or $i=\min J$ and $j\in J\setminus \{i\}$, where $P_J\in \PP$. 
As shown in \cite{CS08}, the intersection $3$-form of $M$ 
may then be written as 
\begin{equation}
\label{eq:mu-bdry}
\mu_M=\sum_{(i,j)\in B} e_{I(i,j)} e_j f_{i,j}\, , 
\end{equation} 
where $I(i,j)=\{k \in [n]\mid  \ell_i\cap \ell_j \in \ell_k\}$ 
and $e_J=\sum_{k\in J} e_k$. 

\begin{theorem}[\cite{CS06, CS08}]
\label{thm:res arr}
If $n\ge 2$ and $\A$ is not a near-pencil, then $\RR^1_1( M)=H^1(M,\C)$. 
\end{theorem}

In depth $k>1$, though, the resonance varieties $\RR^1_k(M)$ may have 
non-linear irreducible components. 

The next result expresses the Alexander polynomial and the first 
characteristic variety of the boundary manifold $M=M(\A)$ in 
terms of the underlying graph  $\Gamma=(V,E)$.

\begin{theorem}[\cite{CS08}]
\label{thm:alex poly arr}
If $\A$ is not a pencil, then
\begin{enumerate} 
\item
$\Delta_{M} = \prod_{v \in V} (t_v-1)^{\deg(v)-2} \in 
\Z[t_1^{\pm 1},\dots,  t_n^{\pm 1}]$, 
where  $t_v=\prod_{i\in v} t_i$ and $t_0\cdots t_n=1$. 
\\[-8pt]
\item 
$\VV^1_1(M) = \bigcup_{v \in V\, :\, \deg(v)\ge 3} 
\set{t_v-1=0}$.
\end{enumerate}
\end{theorem}

Putting now together Corollary \ref{cor:tcone} with the above two 
theorems easily implies the next result. 

\begin{corollary}[\cite{CS08}]
\label{cor:bdry formal}
For the boundary manifold $M$ of a line arrangement $\A$ 
the following conditions are equivalent:
\begin{enumerate}
\item $M$ is formal.
\item $M$ is $1$-formal.
\item $\TC_{\one}(\VV^1_1(M))= \RR^1_1(M)$.
\item $\A$ is either a pencil or a near-pencil.
\end{enumerate}
\end{corollary}

\begin{example}
\label{ex:bdry gen}
Let $\A$ be an arrangement of $n+1\ge 4$ lines in general 
position in $\CP^2$.  Then  $\mu_M=\sum_{1\le i<j\le n} e_i e_j f_{i,j}$ 
and $\RR^1_1(M)=H^1(M,\C)$ properly contains the tangent cone at $\one$ to 
$\VV^1_1(M)=\{\Delta_M=0\}$, 
where $\Delta_M=[(t_1-1)\cdots (t_n-1)(t_1\cdots t_n -1)]^{n-2}$.  

Note that $b_1(M)=\binom{n+1}{2}$, which is an odd integer if $n \equiv 1$ or $2$ mod $4$. 
In this case, the fact that $\TC_{\one}(\VV^1_1(M))\subsetneqq \RR^1_1(M)$ together with 
Theorem \ref{thm:tc 3d}\eqref{odd} imply that $\mu_M$ is not generic.  
\hfill $\Diamond$
\end{example}

\section{Links in the $3$-sphere}
\label{sect:links}

Finally, we analyze the Tangent Cone theorem in the setting 
of knots and links, where the Alexander polynomial originated 
from.  

\subsection{Cohomology ring and resonance varieties}
\label{subsec:links coho}

A link in $S^3$ is a finite collection,  
$L=\{L_1,\dots ,L_n\}$, of disjoint, smoothly embedded  
circles in the $3$-sphere. Let $M=S^3 \setminus \bigcup_{i=1}^n N(L_i)$ 
be the link exterior, i.e., the complement of an open tubular neighborhood 
of $L$. Then $M$ is a compact, connected, orientable 
$3$-manifold, with boundary $\partial M$ consisting of $n$ disjoint tori.  
Furthermore, $M$ is homotopy equivalent to the link complement, 
$X=S^3 \setminus \bigcup_{i=1}^n L_i$.

Picking orientations on the link components yields a preferred basis for 
$H_1(X,\Z)=\Z^n$ consisting of oriented meridians;  let $\{e_1,\dots,e_n\}$ 
be the Kronecker dual basis for $H^1(X,\Z)$.  For each $i\ne j$, 
choose arcs in $X$ connecting $L_i$ to $L_{j}$, and let 
$b_{i,j}\in H^2(X,\Z)$ be their Poincar\'{e}--Lefschetz duals.
Furthermore,  let $\ell_{i,j}=\lk(L_i,L_j)$ be the linking number 
of those two components  (as is well-known, $\ell_{i,j}=\ell_{j,i}$).  
The cohomology ring $H^{\hdot}(X,\Z)$, then, is the 
quotient of the exterior algebra on generators $e_i$ 
and $b_{i,j}$, truncated in degrees $3$ and higher, 
modulo the ideal generated by the relations 
\begin{equation}
\label{eq:coho link}
e_i e_j = \ell_{i,j} b_{i,j} \text{ and }  
b_{i,j} + b_{j,k} + b_{k,i} = 0. 
\end{equation}
In particular, we may choose $\{b_{1,n},\dots ,b_{n-1,n}\}$  as a basis for  
$H^{2}(X,\Z)=\Z^{n-1}$. 

Set $A=H^{\hdot}(X,\C)$ and $S=\C[x_1,\dots, x_n]$, and consider  
the chain complex from \eqref{eq:univ aomoto}.
The $S$-linear map $\delta^1\colon A^1\otimes S\to A^2\otimes S$ 
is given by $\delta^1(e_i)=-\sum_{j=1}^n \ell_{i,j} b_{i,j} \otimes x_j$.  
Rewriting in the chosen basis for $A^2$, we find that the 
transpose matrix, $\partial_2\colon S^{n-1}\to S^n$, has entries
\begin{equation}
\label{eq:del1 link}
(\partial_2)_{i,j}=\ell_{i,j}x_i-\delta_{i,j}\left(\sum_{k=1}^{n}\ell_{i,k}x_k\right).
\end{equation}

The degree $1$ resonance varieties of the link complement, then, 
are the vanishing loci of the codimension $k$ minors of this matrix: 
$\RR^1_k(X)=V(E_k(\partial_2))\subseteq \C^n$. 

\begin{example}
\label{ex:lk1}
If all the linking numbers are equal to $\pm 1$, then 
the cohomology ring is the exterior algebra on $e_1,\dots, e_n$ 
modulo the relations 
$\ell_{i,j}e_ie_j + \ell_{j,k}e_je_k + \ell_{k,i}e_ke_i =0$. 
In the special case when all $\ell_{i,j}$ are equal to $1$, we  
conclude that $\RR^1_1(X)=\{\zero\}$ if $n=2$ and 
$\RR^1_1(X)=\{\sum_{i=1}^{n} x_i = 0\}$ if $n>2$.
On the other hand, if some $\ell_{i,j} = -1$, then the variety $\RR^1_1(X)$   
can be quite complicated, as shown in several examples from \cite[\S6]{MS00}.
\hfill $\Diamond$
\end{example}

\subsection{Characteristic varieties}
\label{subsec:links alex}

Let $\pi=\pi_1(X)$ be the fundamental group of a link complement.  
Using the preferred meridian basis for 
$H_1(X,\Z)=\Z^n$, we may identify the group ring $\Z[\Z^n]$ with 
the ring of Laurent polynomials $\Z[t_1^{\pm 1},\dots , t_n^{\pm 1}]$, 
and view the Alexander polynomial of the link, $\Delta_L=\Delta_{X}$, as 
an element in this ring.  Likewise, we may also identify 
the character group $\Char(X)$ with the algebraic torus $(\C^*)^n$. 
The depth $1$ characteristic variety of the complement, 
$\VV^1_1(X)=\ZZ^1_1(X)$, is a subvariety of $(\C^*)^n$ 
determined by the Alexander polynomial, as follows. 

First suppose that $L$ is a knot, that is, a $1$-component link.  
Then the polynomial $\Delta_{L}\in \Z[t^{\pm 1}]$ 
satisfies $\Delta_{L}(1)=\pm 1$ and 
$\Delta_{L}(t^{-1})\doteq \Delta_{L}(t)$.  In fact, 
every Laurent polynomial satisfying these two conditions 
occurs as the Alexander polynomial of a knot. 
By definition, the Alexander variety $\wV^1_1(X)\subset \C^*$ 
is the set of roots of $\Delta_{L}$; in particular, $1\notin \wV^1_1(X)$.  
On the other hand, $\VV^1_1(X)$ consists 
of all those roots, together with $1$.

Now suppose that the link $L$ has at least two components. Work of Eisenbud 
and Neumann \cite{EN} shows that the first Alexander ideal, $E_1(A_X)$, is 
equal to $I \cdot (\Delta_{L})$, where $I$ is the augmentation ideal of $\Z[\Z^n]$.
Hence, by Proposition \ref{prop:zz1-pi}, 
\begin{equation}
\label{eq:cv link}
\VV^1_1(X)= 
\{z\in (\C^{*})^n \mid \Delta_L(z)=0\}\cup \{\one\}.
\end{equation}

As before, set 
$\widetilde\Delta_L(z_1,\dots,z_n)=\Delta_L(z_1+1,\dots,z_n+1)$. 
The tangent cone to the characteristic variety is then given by the 
following formula:
\begin{equation}
\label{eq:tcone link}
\TC_{\one}(\VV^1_1(X))=
\begin{cases}
V(\init(\widetilde\Delta_L)) & \text{if $\Delta_L(\one)=0$}, \\[2pt]
\{\zero\} & \text{otherwise}.
\end{cases}
\end{equation}

\subsection{Formality}
\label{subsec:formality}

The link complement $X$ has the homotopy type of a $2$-complex;  
thus, $X$ is formal if and only if it is $1$-formal. For a geometrically 
defined class of links (which includes the Hopf links of arbitrarily 
many components) formality holds. 

\begin{example}
\label{ex:algebraic}
Suppose $L$ is an {\em algebraic}\/ link, that is, $\bigcup_{i=1}^n L_i$ 
is the intersection of a complex plane algebraic curve having an isolated 
singularity at a point $p$ with a small $3$-sphere centered at $p$. 
Then, as shown in \cite[Theorem 4.2]{DuH}, the complement $X$ 
is a formal space. 
\end{example}

In general, though, link complements are far from being formal, and, 
in fact, may even fail to admit a $1$-finite $1$-model. 
Their non-formality has been traditionally detected by higher-order 
Massey products. As we shall see below, the resonance varieties 
and the Tangent Cone theorem provide an efficient, alternative way to 
ascertain the non-formality or the non-existence of finite models 
for link complements. 

\subsection{Two-component links}
\label{subsec:links 2comp}

To start with, let $L=\{L_1,L_2\}$ be a $2$-component link, 
and set $\ell=\lk(L_1,L_2)$.  The Alexander polynomial 
$\Delta_L(t_1,t_2)$, then, satisfies the following formula 
due to Torres \cite{Tor}:
\begin{equation}
\label{eq:torres}
\Delta_L(t,1) = \frac{t^{\ell}-1}{t-1}  \Delta_{L_1}(t), 
\end{equation}
and analogously for $\Delta_L(1,t)$. 
Using now the fact that the Alexander polynomial of a knot evaluates 
to $\pm1$ at $1$, we see that 
\begin{equation}
\label{eq:torres1}
\Delta_L(1,1) = \pm \ell.
\end{equation}

\begin{theorem}
\label{thm:2link}
Let $L=\{L_1,L_2\}$ be a $2$-component link, with complement 
$X$.  The following statements are equivalent:
\begin{enumerate}
\item \label{lk1} The space $X$ is formal.
\item \label{lk2} The space $X$ is $1$-formal.
\item \label{lk3} The tangent cone formula 
$\tau_{\one}(\VV^1_1(X))=\TC_{\one}(\VV^1_1(X))=\RR^1_1(X)$ holds.
\item \label{lk4} The linking number $\ell=\lk(L_1,L_2)$ is non-zero.
\end{enumerate}
\end{theorem}

\begin{proof}

As explained previously, implications \eqref{lk1} $\Rightarrow$ \eqref{lk2}  
$\Rightarrow$ \eqref{lk3} hold for arbitrary finite, connected CW-complexes.  

To prove \eqref{lk3} $\Rightarrow$ \eqref{lk4}, 
suppose that $\ell=0$.  Then of course $\RR^1_1(X)=\C^2$. 
On the other hand, the variety $V(\Delta_L)$ 
is an algebraic curve in $(\C^*)^2$.  By \eqref{eq:torres1},
this curve passes through $\one$, and thus, by \eqref{eq:cv link}, 
it coincides with $\VV^1_1(X)$. 
By \eqref{eq:tcone link}, the tangent cone to this variety is the 
algebraic curve in 
$\C^2$ defined by the ideal $\init(\widetilde\Delta_L)$; 
in particular, $\TC_{\one}(\VV^1_1(X))$ is properly contained 
in $\RR^1_1(X)$. 

Finally, to prove \eqref{lk4} $\Rightarrow$ \eqref{lk1}, 
suppose that $\ell\ne 0$.  Then $H^{\hdot}(X,\Q)\cong H^*(T^2,\Q)$, 
and so $X$ is formal. 
\end{proof}

We  illustrate this theorem and related phenomena with several 
examples.  The links which appear in these examples are numbered 
$c^n_r$, where $c$ is the crossing number, $n$ is the number 
of components, and $r$ is the index from Rolfsen's tables \cite{Ro}. 

\begin{example}
\label{ex:link 421}
Let $L$ be the $4^2_1$ link. This is a $2$-component link with 
linking number $2$ and Alexander polynomial $\Delta_L=t_1+t_2$.   
By Theorem \ref{thm:2link}, the complement $X$ is formal, and 
the tangent cone formula \eqref{eq:tc} holds.  
Nevertheless, the variety $\VV^1_1(X)$ has an irreducible 
component (the translated subtorus $t_1t_2^{-1}=-1$), 
which is not detected by the resonance variety $\RR^1_1(X)=\{\zero\}$.
\hfill $\Diamond$
\end{example}

\begin{example}
\label{ex:623-link}
Let $L$ be the $6^2_3$ link.  This is a $2$-component link with 
linking number $2$ and Alexander polynomial $\Delta_L=t_1t_2 -2(t_1+t_2)+1$.   
Again, the complement $X$ is formal, yet $\VV^1_1(X)$ has an irreducible 
component (not passing through $\one$) which this time is 
not a translated algebraic subtorus of $(\C^*)^2$.
\hfill $\Diamond$
\end{example}

\begin{example}
\label{ex:whitehead}
Let $L$ be the $5^2_1$ link, also known as the Whitehead link.  
This $2$-com\-ponent link has linking number $0$; thus, $\RR^1_1(X)=\C^2$ 
and, by Theorem \ref{thm:2link}, $X$ is not formal.  
On the other hand, $\Delta_L=(t_1-1)(t_2-1)$, and so 
$\VV^1_1(X)\subset (\C^*)^2$ consists of the two coordinate subtori, 
$\{t_1=1\}$ and $\{t_2=1\}$.  Consequently 
$\tau_{\one}(\VV^1_1(X))=\TC_{\one}(\VV^1_1(X))=\{x_1=0\}\cup \{x_2=0\}$,    
and so this computation leaves open the question whether $X$ admits a $1$-finite 
$1$-model.
\hfill $\Diamond$
\end{example}

\subsection{Links of many components}
\label{subsec:links ncomp}

We conclude this section with a discussion of links having $3$  
or more components. In the first example, $\RR^1_1(X)$ 
is linear, yet it strictly contains $\tau_{\one}(\VV^1_1(X))$. 

\begin{example}
\label{ex:631}
Let $L$ be the $6^3_1$ link. Then 
$\VV^1_1(X)$ is the subvariety of  $(\C^*)^3$ defined by the polynomial  
$t_1t_2+t_1t_3+t_2t_3-t_1-t_2-t_3$. 
A quick computation shows that $\tau_1(\VV^1_1(X))$ 
is a union of three lines in $\C^3$, namely, 
$\{x_1=x_2+x_3=0\}$,  $\{x_2=x_1+x_3=0\}$, and 
$\{x_3=x_1+x_2=0\}$. 
On the other hand, 
$\TC_{\one}(\VV^1_1(X))=\RR^1_1(X)$ is the plane defined 
by the equation $x_1+x_2+x_3=0$. By Theorem \ref{thm:tcone-fm},  
$X$ admits no $1$-finite $1$-model.
\hfill $\Diamond$
\end{example}

In the next example, the resonance variety $\RR^1_1(X)$ 
is non-linear. 

\begin{example}
\label{ex:842}
Let $L$ be the $8^4_2$ link. Then 
$\VV^1_1(X)\subset (\C^*)^4$ is the zero  locus of the polynomial 
$t_1 t_2 t_3 t_4-t_1 t_2 t_4-t_1 t_3 t_4+t_1 t_3+t_2 t_4-t_2-t_3+1$. 
It follows that $\tau_{\one}(\VV^1_1(X))$ is a union of eight
planes in $\C^4$, 
\begin{align*}
&\{x_1=x_2=0\}\cup  
\{x_3=x_4=0\}\cup 
\{x_1=x_3+x_4=0\}\cup \\
&\: \{x_1+x_2=x_4=0\}\cup
\{x_1-x_2=x_3=0\}\cup 
\{x_2=x_3-x_4=0\}\cup\\
&\:\: \{x_1-x_2+2x_3=x_2-x_3+x_4=0\}\cup
\{x_1-x_2+x_3=2x_2-x_3+x_4=0\}.
\end{align*}
On the other hand, $\TC_{\one}(\VV^1_1(X))=\RR^1_1(X)$ 
is an irreducible quadric, given by the equation 
$(x_1+x_2)x_3=(x_1-x_2)x_4$. 
Hence, once again, $X$ admits no $1$-finite $1$-model. 
\hfill $\Diamond$
\end{example}

An interesting class of examples, studied in detail in \cite{MS00, MS02},  
consists of the singularity links of arrangements of transverse planes in $\R^4$.  
By intersecting such an arrangement $\A=\{H_1,\dots, H_n\}$ with a 
$3$-sphere about $\zero$, we obtain a link $L$ of $n$ great circles in $S^3$. 
The link complement $X$ is aspherical, and its fundamental group 
is a semidirect product of free groups,  $\pi=F_{n-1}\rtimes \Z$. 
If $\A$ is defined by complex equations, then $L$ is the Hopf link, 
and $X$ is formal; in general, though, things are much more complicated.

\begin{example} 
\label{ex:31425} 
Consider the arrangement $\A=\A(31425)$ defined in complex coordinates 
by the function $Q(z,w)=z(z-w)(z-2w)(2z+3w-5\overline{w}) (2z-w-5\overline{w})$. 
The cohomology jump loci of the corresponding link complement $X$ were 
computed in \cite[Example 6.5]{MS00} and \cite[Example 10.2]{MS02}.  
As noted in \cite[Example 8.2]{DPS-duke}, we have that 
$\TC_{\one}(\VV^1_2(X))\subsetneqq \RR^1_2(X)$, and so $X$ is 
not $1$-formal. In fact, $\TC_{\one}(\VV^1_1(X))=\RR^1_1(X)$, yet 
$\tau_{\one}(\VV^1_1(X))\subsetneqq \TC_{\one}(\VV^1_1(X))$, 
and so $X$ admits no $1$-finite $1$-model.
\end{example}

These examples (and many others) lead to the 
following problem regarding the range of applicability of the tangent 
cone formula in the setting of classical links.

\begin{problem}
\label{prb:links}
Given a link complement $X$, determine which (if any) of 
the following equalities is true.
\begin{enumerate}
\item \label{tclk1} $\tau_{\one}(\VV^1_1(X))= \TC_{\one}(\VV^1_1(X))$.
\item \label{tclk2}  $ \TC_{\one}(\VV^1_1(X)) = \RR^1_1(X)$.  
\end{enumerate} 
Does the complement need to admit a $1$-finite $1$-model for the first equality to 
hold?  Does it need to be formal for both equalities to hold?
\end{problem}

\newcommand{\arxiv}[1]
{\texttt{\href{http://arxiv.org/abs/#1}{arXiv:#1}}}
\newcommand{\arxi}[1]
{\texttt{\href{http://arxiv.org/abs/#1}{arxiv:}}
\texttt{\href{http://arxiv.org/abs/#1}{#1}}}
\newcommand{\arxx}[2]
{\texttt{\href{http://arxiv.org/abs/#1.#2}{arxiv:#1.}}
\texttt{\href{http://arxiv.org/abs/#1.#2}{#2}}}
\newcommand{\doi}[1]
{\texttt{\href{http://dx.doi.org/#1}{doi:\nolinkurl{#1}}}}
\renewcommand{\MR}[1]
{\href{http://www.ams.org/mathscinet-getitem?mr=#1}{MR#1}}
\newcommand{\MRh}[2]
{\href{http://www.ams.org/mathscinet-getitem?mr=#1}{MR#1 (#2)}}
\newcommand{\Zbl}[1]
{\href{http://zbmath.org/?q=an:#1}{Zbl #1}}

\end{document}